\documentclass[1 [leqno,11pt]{amsart}
\usepackage{amssymb, amsmath,amsmath,latexsym,amssymb,amsfonts,amsbsy, amsthm,color, CJK}
\usepackage{float}

\numberwithin{equation}{section}

\allowdisplaybreaks

\let\al=\alpha

\let\d=\delta

\let\la=\lambda

\let\f=\frac

\let\om=\omega

\let\D=\Delta

\let\th=\theta
\let\pa=\partial
\let\tre=\triangleq

\def\cF{{\mathcal F}}

\def\tu{\widetilde{u}}
\def\bw{\overline{w}}
\def\bh{\overline{h}}

\def\R{\mathbf R}

\def\N{\mathbf N}

\def\cF{\mathcal F}

\def\D7{\langle D_x\rangle^{\frac 72}}
\def\D{\langle D_x\rangle}
\def\Dxi {\langle \xi \rangle}

\def\eqdef{\buildrel\hbox{\footnotesize def}\over =}

\newcommand{\beq}{\begin{equation}}
\newcommand{\eeq}{\end{equation}}
\newcommand{\ben}{\begin{eqnarray}}
\newcommand{\een}{\end{eqnarray}}
\newcommand{\beno}{\begin{eqnarray*}}
\newcommand{\eeno}{\end{eqnarray*}}

\newtheorem{theorem}{Theorem}[section]

\newtheorem{lemma}[theorem]{Lemma}
\newtheorem{proposition}[theorem]{Proposition}

\newtheorem{remark}[theorem]{Remark}
\newtheorem{Theorem}{Theorem}[section]

\begin{document}

\title[Well-posedness of the Prandtl equation]
{Well-posedness of the linearized Prandtl equation around a non-monotonic shear flow}

\author{Dongxiang Chen}
\address{College of Mathematics and Information Science, Jiangxi Normal University, 330022, Nanchang, P. R. China}
\email{chendx020@aliyun.com}

\author{Yuxi Wang}
\address{School of Mathematical Sciences, Peking University, 100871, Beijing, P. R. China}
\email{wangyuxi0422@pku.edu.cn}

\author{Zhifei Zhang}
\address{School of Mathematical Sciences, Peking University, 100871, Beijing, P. R. China}
\email{zfzhang@math.pku.edu.cn}

\date{\today}

\begin{abstract}
In this paper, we prove the well-posedness of the linearized Prandtl equation around a non-monotonic shear flow  in Gevrey class $2-\th$ for any $\th>0$. This result is almost optimal by the ill-posedness result proved by G\'{e}rard-Varet and Dormy, who construct a class of solution with the growth like $e^{\sqrt{k}t}$ for the linearized Prandtl equation around a non-monotonic shear flow.
\end{abstract}

\maketitle

\section{Introduction}
In this paper, we study the Prandtl equation in $\R_+\times \R_+^2$
\begin{equation}\label{eq:Prandtl-1}
  \left\{
  \begin{aligned}
    &\pa_t u+u \pa_x u +v\pa_{y}u-\pa_{y}^2u+\pa_xp=0,\\
    &\pa_xu+\pa_y v=0,\\
    & u|_{y=0}=v|_{y=0}=0\quad\mbox{and}\quad \displaystyle\lim_{y\to+\infty} u(t,x,y)=U(t,x),\\
     &u|_{t=0}= u_0,
  \end{aligned}
  \right.
\end{equation}
where $(u,v)$ denotes the tangential and normal velocity of the boundary layer flow, and  $(U(t,x), p(t,x))$
is the values on the boundary of the tangential velocity and pressure of the outflow, which satisfies the Bernoulli's law
\[\pa_tU+U\pa_x U+\pa_x p=0.\]
This system introduced by Prandtl \cite{P}  is the foundation of the boundary layer theory. It describes the first order approximation of the velocity field near the boundary in the zero viscosity limit of the Navier-Stokes equations with non-slip boundary condition.  One may check \cite{Olei}  for more introductions on the boundary layer theory.
\smallskip

To justify the zero viscosity limit, one of key step  is to deal with the well-posedness of the Prandtl equation. Due to the lack of horizontal diffusion in (\ref{eq:Prandtl-1}),  the nonlinear term $v\pa_y u$ will lead to one horizontal derivative loss in the process of energy estimate. Up to now, the question of whether the Prandtl equation with general data is well-posed in Sobolev spaces is still open except for some special cases:\smallskip

$\bullet$ Under a monotonic assumption on the tangential velocity of
the outflow, Oleinik \cite{Olei} proved the local
existence and uniqueness of classical solutions to \eqref{eq:Prandtl-1}. With the additional favorable condition on the pressure, Xin and Zhang
\cite{XZ} obtained the global existence of weak solutions to  \eqref{eq:Prandtl-1}.\smallskip

$\bullet$ For the data which is analytic in $x,y$ variables,
Sammartino and Caflisch \cite{SC} established the local well-posedness  of \eqref{eq:Prandtl-1}. Later, the analyticity in $y$ variable was removed by Lombardo, Cannone and Sammartino \cite{LCS}. Zhang and the third author \cite{ZZ} also established the long time well-posedness of (\ref{eq:Prandtl-1}) for small tangential analytic data.\smallskip

Recently,  Alexandre et al. \cite{AWXY} and Masmoudi and Wong \cite{MW} independently develop direct energy method to prove the well-posedness of the Prandtl equation for monotonic data in Sobolev spaces.  Their works might shed some light on the zero viscosity limit problem in Sobolev spaces. See also \cite{KMVW} for the case with multiple monotonicity regions. Recently, we also present an elementary proof by using the paralinearized technique \cite{CWZ}.

On the other hand, G\'{e}rard-Varet and Dormy \cite{GD} proved the ill-posedness in Sobloev spaces for the linearized Prandtl equation around non-monotonic shear flows. The nonlinear ill-posedness was also established in \cite{GN, Guo} in the sense of non-Lipschtiz continuity of the flow. However, G\'{e}rard-Varet and Masmoudi \cite{GM} can prove the well-posedness  of the Prandtl equation (\ref{eq:Prandtl-1})  for a class of data in Gevrey class $\f 74$. In \cite{GM}, the authors conjectured that their result should not be optimal.  The analysis and numerics performed in \cite{GD}  suggest that the optimal exponent may be $s = 2$.
Indeed, G\'{e}rard-Varet and Dormy constructed a class of solution with the growth like $e^{\sqrt{k}t}$ for the linearized Prandtl equation around a non-monotonic shear flow, where $k$ is the tangential frequency.

The goal of this paper is to prove the well-posedness of the linearized Prandtl equation around a non-monotonic shear flow in Gevrey class $2-\th$ for any $\th>0$. This result is almost optimal and in particular implies that the instability mechanism found in \cite{GD} should be severe. The same ideas can be applied to deal with nonlinear Prandtl equation. However, the proof is more involved technically.  So, this will be 
presented in a separable paper in order to present our ideas more clearly here. 
\smallskip

Let $u^s(t,y)$ be the solution of the heat equation
\begin{equation}\label{eq:heat}
  \left\{
  \begin{aligned}
   &\pa_t u^s -\pa_{y}^2u^s=0,\\
    & u^s|_{y=0}=0\quad\mbox{and}\quad \displaystyle\lim_{y\to+\infty} u^s(t,y)=1,\\
     &u^s|_{t=0}= u^s_0(y).
  \end{aligned}
  \right.
\end{equation}
Obviously, $(u^s(t,y),0)$ is a shear flow solution of the Prandtl equation \eqref{eq:Prandtl-1}.
Let us assume that $\pa_yu_0(1)=0$ and for some $c, \delta>0$,
\ben\label{ass:data1}
\begin{split}
&\pa_y^2u_0^s(y)\ge c\quad\text{for}\quad y\in \big[\f12, 2\big],\\
&\big|\pa_yu_0^s(y)\big|\ge c\d e^{-y}\quad \textrm{for}\quad y\in [0,1-\d]\cup\big[1+\d,+\infty).
\end{split}
\een
The linearized Prandtl equation around $(u^s,0)$ takes as follows  
\begin{equation}\label{eq:pran-L}
  \left\{
  \begin{aligned}
    &\pa_t u+u^s\pa_x u+v\pa_{y}u^s-\pa_{y}^2u=0,\\
    &\pa_xu+\pa_y v=0,\\
    &u|_{y=0}=v|_{y=0}=0\quad \mbox{and}\quad\displaystyle\lim_{y\to+\infty}u(t,x,y)=0,\\
    &u(0,x,y)=u_0(x,y).
  \end{aligned}
  \right.
\end{equation}

The main result of this paper is stated as follows.

\begin{Theorem}\label{linearized main thm}
Let $\th\in (0,\f12]$. Assume that  $e^{{\langle D_x\rangle}^{\f12+2\th}}u_0\in H^{\f12,1}_\mu$ with $\pa_y^ku_0|_{y=0}=0$ for $k=0,2$.
Then there exists $T>0$ so that (\ref{eq:pran-L}) has a unique solution $u$ in $[0,T]$, which satisfies \eqref{eq:energy inequ-L}.
In particular, we have
\beno
&u_{\Phi}\in L^\infty(0,T;H^{\f{1}{4}+\th,1}_\mu).
\eeno
Here we denote
\beno
f_\Phi\tre\cF^{-1}\big(e^{\Phi(t,\xi)}\widehat{f}(\xi)\big),\quad \Phi(t,\xi)\tre (1-\lambda t)\Dxi^{\f12+2\theta},
\eeno
and $H^{s,\sigma}_\mu$ is the weighted Sobolev space with $\mu=e^{\f y2}$ which will be introduced later.  
\end{Theorem}

\begin{remark}
Li and Yang \cite{LY} proved the well-posedness of nonlinear Prandtl equation in Gevrey class 2 for the data with non-degenerate critical point and polynomial decay in $y$.
They used G\'{e}rard-Varet and Masmoudi's framework with an introduction of a new unknown $h_1=\pa_y^2u-\f {\pa_y^3u^s} {\pa_y^2u^s}\pa_yu$,
which is  used to control the regularity of $\pa_y^2u$. In the last section, we will explain how to obtain the well-posedness of \eqref{eq:pran-L} in Gevrey class 2 by using our framework and $h_1$. Two methods should be helpful to understand the complex structure of the Prandtl equation and provide evidence about the conjecture that the well-posedness in Gevrey class 2 is optimal. 
\end{remark}

Let us present some key ingredients of our proof.

\begin{itemize}

\item[1.]  Gevrey regularity estimate in monotonic domain. Motivated by \cite{AWXY}, we will introduce the good unknown $w_1=\pa_y\big(\f {u}{\pa_yu^s}\big)$ to control the horizontal regularity of the solution in this domain,
which satisfies
 \beno
 \pa_t w_1+u^s\pa_x w_1 -\pa_{y}^2 w_1=\pa_y F_1.
 \eeno
The key point is that the equation of $w_1$ does not lose the derivative.

\item[2.] Gevrey regularity estimate in non-monotonic domain.
Because $w_1$ does not make sense in non-monotonic domain,
motivated by \cite{GM}, we introduce  $h=d\pa_yu, d=(\pa_y^2u^s)^{-\f12}$ to control the horizontal regularity of the solution in this domain, which satisfies
\begin{align*}
\pa_th+u^s\pa_xh-\pa_y^2h+d(v\pa_y^2 u^s)=&(\pa_t d-\pa_y^2d)\pa_y u-2\pa_yd\pa_y^2 u.
\end{align*}
All the terms in this equation are good except $d(v\pa_y^2u^s)$. 
The key point is
\ben\label{stru}
\int_{\R^2_+}d(v\pa_y^2u^s)hdxdy=0.
\een
So, this term is also good in the energy estimate. However, the localization in $y$ variable will destroy the cancellation structure \eqref{stru}. In particular, the energy estimate in non-monotonic domain will give rise to a new trouble term
\ben
\big(\phi_3(y)\phi_3'(y)v,u\big)_{L^2},\label{trou}
\een
which can be reduced to control the terms like $(w_i,\pa_xu)_{L^2}, i=1,2$ modulus some lower order terms. Here $w_2=\pa_yu^s\pa_yu-\pa_y^2u^su$ and $\phi_3(y)$ is a cut-off function supported in non-monotonic domain. To control them, we need to use the Gevrey regularity and the following

\item[3.] Anisotropic regularity estimates. The unknowns $w_i$ and $h$ have to work in the functional spaces with different horizontal regularity. Roughly speaking,
\beno
 \phi_3(y)h_\Phi\in L^2\big(0,T; H^{\f 14,0}\big),\quad \phi_1(y)(w_1)_\Phi,\, \phi_3(y)(w_2)_\Phi\in L^2(0,T;H^{\f{3}4,0}). 
\eeno 
Here $\phi_1(y)$ is a cut-off function supported in monotonic domain.

\item[4.]  The derivative gain of $w_1$ can be easily obtained by using Gevrey regularity and good structure of $w_1$. The unknown $w_2=\pa_yu^s\pa_yu-\pa_y^2u^su$ satisfies an equation similar to $w_1$, but with a key trouble term in the $H^{\f 12,0}$ energy estimate, which takes 
\beno
\big(\phi_3(y)\pa_y^2u^s\pa_y^2u_\Phi, \phi_3(y)(w_2)_\Phi\big)_{H^{\f12,0}}.
\eeno
The main difficulty is that one can not deduce $\phi_3(y)\pa_y^2u_\Phi\in L^2(0,T;H^{\f 14,0})$ from $\phi_3(y)h_\Phi\in L^2\big(0,T; H^{\f 14,0}\big)$. However, one can prove  $\phi_3(y)(w_2)_\Phi\in L^2(0,T;H^{\f{5}8,0})$ by using some key structures found in \cite{GM}. As we said above, this estimate is not enough to handle \eqref{trou}. On the other hand, one can prove the same regularity as $w_1$ in the framework of Gevrey class $\f 74$. This may be the main reason why the work \cite{GM} can achieve the well-posedness in Gevrey class $\f 74$.

\item[5.] Improved regularity estimate of $w_2$. Compared with $w_1$, $w_2$ lose $\f 18$-order derivative. However, we find that
$\varphi^{1+\th_1}\D^{\f {3} 4+\th}(w_2)_\Phi\in L^2(0,T;L^2)$, where $\varphi$ is a cut-off function vanishing at critical point.
Compared with the work \cite{GM}, this weighted estimate is completely new, and moreover is enough to handle \eqref{trou}.
The price to pay is to use Gevrey $2-\th$ regularity. 

\item[6.] In our framework, if we use the unknown $h_1$, we can easily deduce $\phi_3(y)\pa_y^2u_\Phi \in L^2(0,T;H^{\f 14,0})$, thus $\phi_3(y)(w_2)_\Phi\in  L^2(0,T;H^{\f{3}4,0})$ and avoid the Gevrey regularity loss.
This will be explained in the last section.

\end{itemize}

Let us conclude the introduction with the following notations. 
Let $\om(y)$ be a nonnegative function in $\R^+$.
We introduce the weighted $L^p$ norm
\beno
&&\|f\|_{L^p_\om}\eqdef \|\om(y)f(x,y)\|_{L^p},
\quad\|f\|_{L^p_{y,\om}}\eqdef \|\om(y)f(y)\|_{L^p}.
\eeno
The weighted anisotropic Sobolev space $H^{s,\ell}_\om$ for $s=k+\sigma$ and $k,\ell\in \N, \sigma\in [0,1)$ consists of all functions $f\in L^2_\om$ satisfying
\beno
\|f\|_{H^{s,\ell}_\om}^2\eqdef \sum_{\al\le k}\sum_{\beta\le \ell}\|\pa_x^\al\D^\sigma\pa_y^\beta f\|_{L^2_\om}^2<+\infty.
\eeno
We denote by $H^\ell_{y,\om}$ the weighted Sobolev space in $\R_+$, which consists of all functions $f\in L^2_{y,\om}$ satisfying
\beno
\|f\|_{H^{\ell}_{y,\om}}^2\eqdef \sum_{\beta\le \ell}\|\pa_y^\beta f\|_{L^2_{y,\om}}^2<+\infty.
\eeno
In the case when $\om=1$, we denote $H^{k,\ell}_\om$ by $H^{k,\ell}$, and $H^\ell_{y,\om}$ by $H^\ell_y$ for the simplicity.

\section{Basic estimates for the shear flow}

Let $u^s(t,y)$ be the solution of the heat equation
\begin{equation}\label{eq:shear}
  \left\{
  \begin{aligned}
   &\pa_t u^s -\pa_{y}^2u^s=0,\\
    & u^s|_{y=0}=0\quad\mbox{and}\quad \displaystyle\lim_{y\to+\infty} u^s(t,y)=1,\\
  &u^s|_{t=0}= u^s_0(y).
  \end{aligned}
  \right.
\end{equation}
\begin{proposition}\label{prop:shear}
Assume that $\pa_y u_0^s\in H^3_{y,\mu}$ and $u_0^s(0)=0, \pa_y^2u_0^s(0)=0$. Then it holds that
for any $t\in [0,+\infty)$,
\begin{align*}
E^s(t)\eqdef \|\pa_y u^s(t)\|_{H^{3}_{y,\mu}}^2+\int_0^t\|\pa_yu^s(\tau)\|_{H^4_{y,\mu}}^2d\tau
\le \|\pa_yu_0^s\|_{H^3_{y,\mu}}^2e^{Ct}.
\end{align*}
Moreover, if for $k=0,1,2,3$,
\beno
|\pa_y^k\big(u^s_0-1\big)(y)|\leq c^{-1}e^{-y}\quad\quad \textrm{for}\quad y\in [0,+\infty),
\eeno
then we have
\beno
&&|\pa_y^k(u^s(t,y)-1)|\leq  Ce^{-y},
\eeno
for $(t,y)\in[0,1]\times [0,+\infty).$
\end{proposition}
\begin{proof}
Taking $L^2_{y,\mu}$ inner product between the first equation of (\ref{eq:shear}) and $u^s_t$ , we obtain
\beno
\f d {dt}\|\pa_yu^s\|_{L^2_{y,\mu}}^2+\|u^s_t\|_{L^2_{y,\mu}}^2\le C\|\pa_yu^s\|_{L^2_{y,\mu}}^2.
\eeno
Taking the time derivative to the first equation of (\ref{eq:shear}), then taking $L^2_{y,\mu}$ inner product between the resulting equation and $u^s_t$, we get
\beno
&&\f d {dt}\|u^s_t\|_{L^2_{y,\mu}}^2+\|\pa_yu^s_t\|_{L^2_{y,\mu}}^2\le C\|u^s_t\|_{L^2_{y,\mu}}^2.
\eeno
And taking $L^2_{y,\mu}$ inner product between the resulting equation and $\pa_y^2u^s_t$, we get
\begin{align*}
\f d{dt}\|\pa_y\pa_tu^s\|_{L_{y,\mu}^2}+\|\pa_y^2\pa_tu^s\|_{L_{y,\mu}^2}^2\le C\|\pa_y\pa_tu^s\|_{L_{y,\mu}^2}^2+\f12\|\pa_t^2u^s\|_{L_{y,\mu}^2}^2.
\end{align*}
Taking the $\pa_t\pa_y$ to the first equation of (\ref{eq:shear}), then taking $L^2_{y,\mu}$ inner product between the resulting equation and $\pa_t\pa_y^3u^s$, we deduce that
\begin{align*}
\f d{dt}\|\pa_y^2\pa_tu^s\|_{L_{y,\mu}^2}+\|\pa_{t}\pa_y^3u^s\|_{L_{y,\mu}^2}^2\le C\|\pa_y^2\pa_tu^s\|_{L_{y,\mu}^2}^2+\f12\|\pa_y\pa_t^2u^s\|_{L_{y,\mu}^2}^2.
\end{align*}
Using $\pa_tu^s=\pa_y^2u^s$, we deduce from Gronwall's inequality that
\beno
&&\|\pa_yu^s(t)\|_{L^2_{y,\mu}}^2+\|u^s_t(t)\|_{L^2_{y,\mu}}^2+\|\pa_y\pa_tu^s\|_{L_{y,\mu}^2}+\|\pa_y^2\pa_tu^s\|_{L_{y,\mu}^2}^2
\le \|\pa_yu_0^s\|_{H^3_{y,\mu}}^2e^{Ct},
\eeno
from which and $\pa_tu^s=\pa_y^2u^s$, it follows that
\beno
E^s(t)\le \|\pa_yu_0^s\|_{H^3_{y,\mu}}^2e^{Ct}.
\eeno

For the  pointwise estimates, we need to use the representation formula of the solution
\begin{align*}
u^s(t,y)&=\frac1{2\sqrt{\pi t}}\int_0^{+\infty}\Big(e^{-\frac{(y-y')^2}{4t}}-e^{-\frac{(y+y')^2}{4t}}\Big)u_0^s(y')dy'.
\end{align*}
We write
\begin{align*}
u^s(t,y)-1=&-\f{1}{2\sqrt{\pi t}}\int_0^{+\infty}e^{-\f{(y+y')^2}{4t}}u^s_0(y')dy'+\f{1}{2\sqrt{\pi t}}\int_0^{+\infty}e^{-\f{(y-y')^2}{4t}}(u^s_0(y')-1)dy'\\
&+\Big( \f{1}{2\sqrt{\pi t}}\int_0^{+\infty}e^{-\f{(y-y')^2}{4t}}dy'-1\Big)\\
\triangleq&I_1+I_2+I_3.
\end{align*}
The result is obvious for $|y|\le 4$. So, we assume $y\ge 4\ge 4t$.
Thanks to  $|u_0^s(y)|\leq C$, it follows that
\begin{align*}
|I_1|\leq& \f{1}{2\sqrt{\pi t}}\int_0^{+\infty}e^{-\f{y^2}{4t}}~e^{-\f{2yy'+(y')^2}{4t}}|u^s_0(y')|dy'\\
\leq&\f{1}{2\sqrt{\pi t}}e^{-\f{y^2}{4t}}\int_0^{+\infty}e^{-\f{(y')^2}{4t}}|u^s_0(y')|dy'\\
\leq &Ce^{-y}.
\end{align*}
Thanks to $|u_0^s(y)-1|\leq c^{-1}e^{-y}$, we infer that
\begin{align*}
|I_2|\leq& e^{-y} \f{1}{2\sqrt{\pi t}}\int_0^{+\infty} e^{-\f{(y-y')^2}{4t}}e^{y-y'}~|u^s_0(y')-1|~e^{y'}dy'\\
\leq&
Ce^{-y}\f{1}{2\sqrt{\pi t}}\int_{0}^{+\infty}e^{-\f{y'^2}{4t}}e^{y'}dy'\\
\leq&
Ce^{-y}e^t\int_{0}^{+\infty}e^{-(\xi-\sqrt{t})^2}d\xi\\
\leq& Ce^{-y}.
\end{align*}
For $I_3$, we have
\begin{align*}
|I_3|\leq& \Big|\f{1}{\sqrt{\pi }}\int_{-\f{y}{2\sqrt{t}}}^{+\infty} e^{-\xi^2}d\xi-1\Big|=\f{1}{\sqrt{\pi }}\int_{\f{y}{2\sqrt{t}}}^{+\infty} e^{-\xi^2}d\xi.
\end{align*}
If $2\sqrt t\leq 1$, then
\begin{align*}
|I_3|\leq C \int_{\f{y}{2\sqrt{t}}}^{+\infty} e^{-\xi}d\xi\leq C e^{-\f{y}{2\sqrt t}}\leq Ce^{-y},
\end{align*}
and if $2\sqrt t\geq 1$ and $y\geq 4t$, then
\begin{align*}
|I_3|\leq C \int_{\f{y}{2\sqrt{t}}}^{+\infty} e^{-2\sqrt{t}\xi}d\xi \leq \f{C}{2\sqrt{t}}e^{-y}\leq Ce^{-y}.
\end{align*}
Putting the estimates of $I_1-I_3$ together,  we deduce that
\beno
|u^s(t,y)-1|\le Ce^{-y}.
\eeno

Thanks to $u_0^s(0)=0$ and $\pa_y^2u_0^s(0)=0$, we get by integration by parts that
\begin{align*}
&\pa_{y} u^s(t,y)=\frac1{2\sqrt{\pi t}}\int_0^{+\infty}\Big(e^{-\frac{(y-y')^2}{4t}}+e^{-\frac{(y+y')^2}{4t}}\Big)\pa_yu_0^s(y')dy',\\
&\pa_y^2u^s(t,y)=\frac1{2\sqrt{\pi t}}\int_0^{+\infty}\Big(e^{-\frac{(y-y')^2}{4t}}-e^{-\frac{(y+y')^2}{4t}}\Big)\pa_y^2u_0^s(y')dy',\\
&\pa_y^3u^s(t,y)=\frac1{2\sqrt{\pi t}}\int_0^{+\infty}\Big(e^{-\frac{(y-y')^2}{4t}}+e^{-\frac{(y+y')^2}{4t}}\Big)\pa_y^3u_0^s(y')dy'.
\end{align*}
Then in the same derivation as in $I_2$, we have for $k=1,2,3,$
\beno
|\pa_y^ku^s(t,y)|\le Ce^{-y} \quad \text{for}\, (t,y)\in [0,1]\times [0,+\infty).
\eeno

This finishes the proof of the proposition.
\end{proof}

\begin{lemma}\label{lem:shear}
Let $u_0^s(y)$ be as in Proposition \ref{prop:shear}. If $u_0^s(y)$ satisfies \eqref{ass:data1}, then
there exists $T_1>0$ so that for any $t\in [0,T_1]$,
\beno
&&\pa_y^2u^s(t,y)\ge \f c2\quad\text{for}\quad y\in \big[\f12, 2\big],\\
&&\pa_yu^s(t,y)\ge \f c 2\d e^{-y}\quad \textrm{for}\quad y\in [0,1-\d]\cup\big[1+\d,+\infty).
\eeno
\end{lemma}

\begin{proof}
We have
\beno
&&\pa_yu^s(t,y)=\pa_yu_0^s(y)+\int_0^t\pa_t\pa_yu^s(\tau,y)d\tau,\\
&&\pa_y^2u^s(t,y)=\pa_y^2u_0^s(y)+\int_0^t\pa_t\pa_y^2u^s(\tau,y)d\tau.
\eeno
Notice that
\beno
&&\Big|\int_0^t\pa_t\pa_yu^s(\tau,y)d\tau\Big|\le Ct^\f12\|\pa_y^3u^s\|_{L^2_tH^1_y},\\
&&\Big|\int_0^t\pa_t\pa_y^2u^s(\tau,y)d\tau\Big|\le Ct^\f12\|\pa_y^4u^s\|_{L^2_tH^1_y}.
\eeno
Then the lemma follows from Proposition \ref{prop:shear} and \eqref{ass:data1}.
\end{proof}

\section{Introduction of good unknowns}

An essential difficulty solving the Prandtl equations is the loss of one derivative in the horizontal direction $x$ induced by the term $v\pa_yu^s$. To eliminate the trouble term $\pa_y u^s v$ in (\ref{eq:pran-L}),  it is natural to introduce a good unknown $w_1$ defined by
\beno
w_1\eqdef\pa_y\big(\f{u}{\pa_y u^s}\big),
\eeno
which is motivated by the work \cite{AWXY}.
Then a direct calculation gives
\begin{equation}\label{eq: linearized Prandtl-w_1}
  \left\{
  \begin{aligned}
    &\pa_t w_1+u^s\pa_x w_1 -\pa_{y}^2 w_1=\pa_y F_1,\\
    & \pa_yw_1|_{y=0}=0\quad \mbox{and}\quad\displaystyle\lim_{y\to+\infty} w_1=0,\\
     &w_1|_{t=0}= w_0(x,y),
  \end{aligned}
  \right.
\end{equation}
where $F_1$ is given by
\beno
&&F_1=u\pa_t\big(\f{1}{\pa_y u^s}   \big)-\big[\pa_y^2,\f{1}{\pa_y u^s}\big]u.
\eeno
Here we used the fact that $\pa_y^2 u=0$ on $y=0$, which can be seen from (\ref{eq:pran-L}).

Notice that $w_1$ is only well-defined in the monotonic domain. While, Lemma \ref{lem:shear} tells us
\ben\label{ass: linear mono}
|\pa_yu^s(t,y)|\ge \f {c\d} 2 e^{-y}\quad \text{for}\quad (t,y)\in [0,T_1]\times \big([0,1-\d]\cup[1+\d,+\infty)\big).
\een
Then it is natural to introduce a cut-off good known
\begin{align*}
&\bw_1\triangleq e^{-\f{y}{2}}\phi_1(y)\pa_y\big(\f{u}{\pa_yu^s}\big)\tre\psi_1(y)\pa_y(\f{u}{\pa_yu^s}),
\end{align*}
where $\phi_1(y)\in C^\infty(\mathbb{R}_{+})$ with the support included
in $[0,1-\d]\cup[1+\d,+\infty)$ and $\phi_1(y)=1$ in $\big[0,1-2\d\big]\cup\big[1+2\d,+\infty\big]$. A direct calculation shows
\begin{align}\label{eq:w_1}
&\pa_t\bw_1+u^s\pa_x\bw_1-\pa_y^2\bw_1=[\psi_1(y),\pa_y^2]w_1+\psi_1(y)\pa_yF_1.
\end{align}

To control the regularity of the solution  in the non-monotonic domain,  we need to use the non-degenerate condition
\begin{align}\label{ass:non}
\pa_y^2u^s(t,y)\ge \f c 2\quad\mbox{for}\quad(t,y)\in[0,T_1]\times [\f12,2].
\end{align}
Motivated by \cite{GM}, we introduce a good unknown $h$ defined by
\begin{align*}
h\eqdef d\pa_y u,
\end{align*}
where $d(t,y)=\phi_3(y)(\pa_y^2u^s)^{-1/2}$ and $\phi_3(y)$ is  a cut-off function supported in  $[\f12,2]$ and $\phi_3(y)=1$ as $y\in[\f34,\f74]$.
Then $h$ satisfies
\begin{align}\label{eq:h}
\pa_th+u^s\pa_xh-\pa_y^2h+d(v\pa_y^2 u^s)=&(\pa_t d-\pa_y^2d)\pa_y u-2\pa_yd\pa_y^2 u.
\end{align}

To propagate the regularity of the solution from monotonic domain to non-monotonic domain, we need to introduce another good unknown $\bw_2$
\ben
\bw_2\eqdef \psi_2(y)\big(\pa_y u^s\pa_yu-u\pa_y^2 u^s\big)\tre \psi_2(y)w_2,
\een
where  $\psi_2(y)\in C_0^\infty(\mathbb{R}_{+})$ with the support included
in $[1-3\d, 1+3\d]$ and $\psi_2(y)=1$ in $\big[1-2\d, 1+2\d]$.
It is easy to check that
\begin{align}
&\pa_t\bw_2+u^s\pa_x\bw_2-\pa_y^2\bw_2=[\psi_2(y),\pa_y^2]w_2+\psi_2(y)F_2,\label{eq:w2}
\end{align}
where
\begin{align*}
F_2=\pa_t\pa_y u^s\pa_yu+[\pa_y u^s,\pa_y^2]\pa_y u-u\pa_t\pa_y^2 u^s-[\pa_y^2 u^s,\pa_y^2]u.
\end{align*}
In fact, $w_1$ and $w_2$ are basically equivalent in the monotonic domain by the relation
\begin{align*}
w_2=(\pa_y u^s)^2w_1.
\end{align*}
This in particular implies that

\begin{lemma}\label{lem: linearized w12-1}
It holds that
\begin{align*}
&\|1_{I_1}(y)(w_1)_\Phi\|_{H^{\f 12,0}}\leq C\|(\bw_2)_\Phi\|_{H^{\f{1}{2},0}},\\
&\|1_{I_2}(y)(w_2)_\Phi\|_{H^{\f{1}{2},0}}\leq C\|(\bw_1)_\Phi\|_{H^{\f 12,0}},
\end{align*}
where $I_1=\text{supp}\phi_1'$ and $I_2=\text{supp}\psi_2'$.
\end{lemma}

Let $a(t)$  be a critical point of $u^s(t,y)$, i.e.,
\beno
\pa_yu^s(t,a(t))=0.
\eeno
Therefore, $a(t)$ satisfies
\beno
\pa_ta(t)=-\f {\pa_t\pa_yu^s(t,a)} {\pa_y^2u^s(t,a)},\quad a(0)=1.
\eeno
By Proposition \ref{prop:shear}, there exists $T_2>0$ so that
\ben\label{ass:critical}
|a(t)-1|\le 2\d\quad\text{for}\quad t\in [0,T_2].
\een
Then $u$ can be represented in terms of $w_1, w_2$. More precisely,

\begin{lemma}\label{lem:u-decom}
We can decompose $u$ as $u=u_1+u_2$, where
\begin{align*}
u_1=
\left\{
\begin{aligned}
&\pa_y u^s\int_{0}^{y}\phi_1w_1dy'\quad\text{for}\quad y<1-2\d,\\
&\pa_y u^s\Big(\int_{0}^{1-2\d}\phi_1w_1dy'+\int_{1-2\d}^{y'}\f{\bw_2}{(\pa_y u^s)^2}dy'\Big)\quad\text{for}\quad 1-2\d \leq y<a(t),\\
&\pa_y u^s\Big(\int_{1+2\d}^y\f{\bw_2}{(\pa_y u^s)^2}dy'+\int_{2}^{1+2\d}\phi_1w_1dy'\Big)\quad\text{for}\quad a(t) < y<1+2\d,\\
&\pa_y u^s\int_{2}^{y}\phi_1w_1dy' \quad\text{for}\quad y\geq 1+2\d,
\end{aligned}
\right.
\end{align*}
and
\begin{align*}
u_2=\pa_y u^s \f{u(t,x,2)}{\pa_y u^s(t,2)}1_{\{y>a(t)\}}(y).
\end{align*}
\end{lemma}

\section{Gevrey regularity estimate of $\bw_1$}

In what follows, let us always assume that $T\le\min(T_1,T_2)$.

\begin{proposition}\label{prop:w1-L}
Let $\bw_1$ be a smooth solution of (\ref{eq:w_1})  in $[0,T]$. Then it holds that for any $t\in [0,T]$,
\begin{align*}
&\f d{dt}\|(\bw_1)_\Phi\|_{H^{\f{1}{2},0}}^2+(\lambda-C)\|(\bw_1)_\Phi\|_{H^{\f{3}{4}+\th,0}}^2+\|\pa_y (\bw_1)_\Phi\|_{H^{\f{1}{2},0}}^2\\
&\quad\le C\Big(\|u_\Phi\|_{H^{\f{1}{4},1}_\mu}^2+\|(\bw_1)_\Phi\|_{H^{\f{1}{2},0}}^2+\|(\bw_2)_\Phi\|_{H^{\f{1}{2},0}}^2\Big).
\end{align*}
\end{proposition}

Let us  begin with the estimates of source term $F_1$.

\begin{lemma}\label{lem: linear F-S}
It holds that
\begin{align*}
&\|\psi_1(y)\pa_y(F_1)_{\Phi}\|_{H^{\f{1}{4},0}}\le
 C\big(\|u_\Phi\|_{H^{\f{1}{4},1}_\mu}+\|(\bw_1)_\Phi\|_{H^{{\f{1}{2}},1}}+\|(\bw_2)_\Phi\|_{H^{{\f{1}{2}},0}}\big).
\end{align*}
\end{lemma}

\begin{proof}
An easy calculation gives
\begin{align*}
F_1=-2\f{(\pa_y^2 u^s)^2}{(\pa_y u^s)^2}u+2\f{\pa_y^2 u^s}{(\pa_y u^s)^2}\pa_yu.
\end{align*}
By (\ref{ass: linear mono}), we get
\begin{align*}
\|\psi_1(y)\pa_y(F_1)_{\Phi}\|_{H^{\f{13}{4},0}}
\leq C\|u_\Phi\|_{H^{\f{1}{4},1}_\mu}+C\|\psi_1(y)e^y(\pa_y^2 u)_{\Phi}\|_{H^{\f{1}{4},0}}.
\end{align*}
Notice that
\begin{align*}
\psi_1(y)\pa_y^2 u=\pa_y u^s\Big(\pa_y \bw_1-\psi_1'w_1+2\psi_1\f{\pa_y u\pa_y^2 u^s}{(\pa_y u^s)^2}+\psi_1\pa_y(\f{\pa_y^2 u^s}{(\pa_y u^s)^2})u   \Big),
\end{align*}
which along with Lemma \ref{lem:shear} implies that
\begin{align*}
\|\psi_1(y)e^y(\pa_y^2 u)_{\Phi}\|_{H^{\f{1}{4},0}}\leq
C\big(\|u_\Phi\|_{H^{\f{1}{4},1}_\mu}+\|(\bw_1)_\Phi\|_{H^{{\f{1}{2}},1}}+\|(\bw_2)_\Phi\|_{H^{{\f{1}{2}},0}}\big).
\end{align*}

Putting the above estimates together, we conclude the lemma.
\end{proof}

Now we are in position to prove proposition \ref{prop:w1-L}.
\begin{proof}
Applying $e^{\Phi(t,D_x)}$ to (\ref{eq:w_1}), we obtain
\begin{align}\label{ linearized Prandtl-tbw_1_phi}
\nonumber
\pa_t(\bw_1)_{\Phi}+\lambda\langle D_x\rangle^{\f 12+2\th}(\bw_1)_{\Phi}+&u^s\pa_x(\bw_1)_{\Phi}-\pa_y^2(\bw_1)_{\Phi}\\
&=[\psi_1(y),\pa_y^2](w_1)_{\Phi}+\psi_1(y)\pa_y(F_1)_{\Phi}.
\end{align}

Making $H^{\f{1}{2},0}$ energy estimate to  (\ref{ linearized Prandtl-tbw_1_phi}), we obtain
\begin{align*}
&\frac12\frac d{dt}\|(\bw_1)_\Phi\|_{H^{\f{1}{2},0}}^2+\lambda\|(\bw_1)_\Phi\|_{H^{\f{3}{4}+\th,0}}^2-\big(\pa_{y}^2 (\bw_1)_\Phi,(\bw_1)_\Phi\big)_{H^{\f{1}{2},0}}+\big(u^s\pa_x( \bw_1)_\Phi,(\bw_1)_\Phi\big)_{H^{\f{1}{2},0}}\\
&\quad\quad=\big([\psi_1(y),\pa_y^2](w_1)_{\Phi},(\bw_1)_\Phi\big)_{H^{\f{1}{2},0}}+\big(\psi_1(y)\pa_y(F_1)_{\Phi}, (\bw_1)_\Phi\big)_{H^{\f{1}{2},0}}.
\end{align*}

Thanks to $\pa_y (\bw_1)_\Phi|_{y=0}=0$ , we get by integration by parts that
\begin{align*}
&-(\pa_{y}^2 (\bw_1)_\Phi,(\bw_1)_\Phi)_{H^{\f{1}{2},0}}=\|\pa_y (\bw_1)_\Phi\|_{H^{\f{1}{2},0}}^2,\quad\big(u^s\pa_x( \bw_1)_\Phi,(\bw_1)_\Phi\big)_{H^{\f{1}{2},0}}=0.
\end{align*}
We infer from Lemma \ref{lem: linearized w12-1}  that
\begin{align*}
&\big([\psi_1(y),\pa_y^2](w_1)_{\Phi},(\bw_1)_\Phi\big)_{H^{\f{1}{2},0}}\\
&\leq
2\big|\big(\psi_1'(w_1)_{\Phi},\pa_y(\bw_1)_\Phi\big)_{H^{\f{1}{2},0}}\big|
+2\big|\big(\psi_1''(w_1)_{\Phi},(\bw_1)_\Phi\big)_{H^{\f{1}{2},0}}\big|\\
&\leq
C\|(\bw_2)_\Phi\|_{H^{\f{1}{2},0}}\big(\|\pa_y (\bw_1)_\Phi\|_{H^{\f{1}{2},0}}+\|(\bw_1)_\Phi\|_{H^{\f{1}{2},0}}\big)\\
&\leq
C\big(\|(\bw_2)_\Phi\|_{H^{\f{1}{2},0}}^2+\|(\bw_1)_\Phi\|_{H^{\f{1}{2},0}}^2\big)+\f{1}{8}\|\pa_y (\bw_1)_\Phi\|_{H^{\f{1}{2},0}}^2.
\end{align*}
It follows from  Lemma \ref{lem: linear F-S}  that
\begin{align*}
\big(\psi_1(y)\pa_y(F_1)_{\Phi}, (\bw_1)_\Phi\big)_{H^{\f{1}{2},0}}
\le& C\big(\|u_\Phi\|_{H^{\f{1}{4},1}_\mu}^2+\|(\bw_1)_\Phi\|_{H^{{\f{1}{2}},0}}^2+\|(\bw_2)_\Phi\|_{H^{{\f{1}{2}},0}}^2\big)\\
&+\f18\|\pa_y(\bw_1)_\Phi\|_{H^{{\f{1}{2}},0}}^2+C\|(\bw_1)_\Phi\|_{H^{{\f{3}{4}},0}}^2.
\end{align*}

Summing up all the estimates, we conclude  the proposition.
\end{proof}

\section{Gevrey regularity estimate of $\bw_2$}

First of all, we prove Gevrey regularity without weight. 

\begin{proposition}\label{prop:w2}
Let $\bw_2$ be a  solution of (\ref{eq:w2})  in $[0,T]$.
There exists $\delta>0$ small enough so that for any $t\in [0,T]$,
\begin{align*}
&\f d{dt}\|(\bw_2)_\Phi\|_{H^{\f{3}{8},0}}^2+(\lambda-C)\|(\bw_2)_\Phi\|_{H^{\f{5}{8}+\th,0}}^2+\|\pa_y (\bw_2)_\Phi\|_{H^{\f{3}{8},0}}^2\\
&\quad\le C\Big(\|u_\Phi\|_{H^{\f{1}{4},1}}^2+\|(\bw_1)_\Phi\|_{H^{\f 12,0}}^2+\|(\bw_2)_\Phi\|_{H^{\f{3}{8},0}}^2\Big).
\end{align*}
\end{proposition}

The proposition can be proved by following the proof of Proposition \ref{prop:w1-L} and using the following lemma.

\begin{lemma}\label{lem:F2-L}
It holds that
\begin{align*}
\big(\psi_2(y)(F_2)_\Phi,(\bw_2)_\Phi\big)_{H^{\f{3}{8},0}}
\le&  C\big(\|u_\Phi\|_{H^{\f{1}{4},1}}^2+\|(\bw_1)_\Phi\|_{H^{\f 12,0}}^2+\|(\bw_2)_\Phi\|_{H^{\f{3}{8},0}}^2+\|(\bw_2)_\Phi\|_{H^{\f{5}{8},0}}^2\big)\\
&+\f12\|\pa_y(\bw_2)_\Phi\|_{H^{{\f{3}{8}},0}}.
\end{align*}
\end{lemma}

\begin{proof}
Notice that
\begin{align}
-u\pa_t\pa_y^2 u^s-[\pa_y^2 u^s,\pa_y^2]u=2\pa_y^3 u^s \pa_y u,\label{eq:com1}
\end{align}
therefore,
\begin{align*}
\big(\psi_2(u\pa_t\pa_y^2 u^s+[\pa_y^2 u^s,\pa_y^2]u)_\Phi,(\bw_2)_\Phi\big)_{H^{\f{3}{8},0}}\leq C\|u_\Phi\|_{H^{\f{1}{4},1}_\mu}\|(\bw_2)_\Phi\|_{H^{\f{5}{8},0}}.
\end{align*}
Similarly, we have
\begin{align*}
\big(\psi_2(\pa_t\pa_y u^s\pa_yu+[\pa_y u^s,\pa_y^2]\pa_y u)_\Phi,(\bw_2)_\Phi\big)_{H^{\f{3}{8},0}}=-2\big(\psi_2\pa_y^2 u^s \pa_y^2 u_\Phi,(\bw_2)_\Phi\big)_{H^{\f{3}{8},0}}.
\end{align*}
The estimate of this term is very tricky. The following argument was motivated by \cite{GM}.
By Lemma \ref{lem:u-decom}, $\pa_y u$ can be written as $\pa_y u=\pa_y u_1+\pa_y u_2 $. Note that both $\pa_y u_1$ and $\pa_y u_2$ are discontinuous across $y=a(t).$ In particular, we have
\beno
\lim_{y\rightarrow a(t)-}\pa_yu_1-\lim_{y\rightarrow a(t)+}\pa_yu_1=\pa_y^2u^s\f {u(t,x,2)} {\pa_yu^s(t,2)}\tre J.
\eeno
Then by integration by parts, we get
\begin{align*}
-2\big(\psi_2(y)(\pa_y^2 u^s \pa_y^2 u)_\Phi,(\bw_2)_\Phi\big)_{H^{\f{3}{8},0}}=&-2\int_{y>a(t)}\psi_2(y)\pa_y^2 u^s\langle D_x\rangle^{\f{3}{8}}(\pa_y^2 u_1)_\Phi \langle D_x\rangle^{\f{3}{8}}(\bw_2)_\Phi dxdy\\
&-2\int_{y<a(t)}\psi_2(y)\pa_y^2 u^s\langle D_x\rangle^{\f{3}{8}}(\pa_y^2 u_1)_\Phi \langle D_x\rangle^{\f{3}{8}}(\bw_2)_\Phi dxdy\\
&-2\int_{y>a(t)}\psi_2(y)\pa_y^2 u^s\langle D_x\rangle^{\f{3}{8}}(\pa_y^2 u_2)_\Phi \langle D_x\rangle^{\f{3}{8}}(\bw_2)_\Phi dxdy\\
=&2\int_{\R^2_{+}}\psi_2'(y)\pa_y^2 u^s\langle D_x\rangle^{\f{3}{8}}(\pa_y u_1)_\Phi \langle D_x\rangle^{\f{3}{8}}(\bw_2)_\Phi dxdy\\
&+2\int_{\R^2_{+}}\psi_2(y)\pa_y^3 u^s\langle D_x\rangle^{\f{3}{8}}(\pa_y u_1)_\Phi \langle D_x\rangle^{\f{3}{8}}(\bw_2)_\Phi dxdy\\
&+2\int_{\R^2_{+}}\psi_2(y)\pa_y^2 u^s\langle D_x\rangle^{\f{3}{8}}(\pa_y u_1)_\Phi \langle D_x\rangle^{\f{3}{8}}(\pa_y\bw_2)_\Phi dxdy\\
&-2\int_{y>a(t)}\psi_2(y)\pa_y^2 u^s\langle D_x\rangle^{\f{3}{8}}(\pa_y^2 u_2)_\Phi \langle D_x\rangle^{\f{3}{8}}(\bw_2)_\Phi dxdy\\
&-2\int_{y=a(t)}\pa_y^2 u^s \langle D_x\rangle^{\f{3}{8}}J_\Phi\langle D_x\rangle^{\f{3}{8}}(\bw_2)_\Phi dx\\
\triangleq& A_1+\cdots+A_5.
\end{align*}

Note that $\psi_2'(y)$ vanishes in a neighborhood of $y=a(t)$ so that $\pa_y u_1$ behaves like $w_1$ on the support of $\psi_2'(y).$ Thus,
\begin{align*}
A_1\leq C\|(\bw_1)_\Phi\|_{H^{\f12,0}}\|(\bw_2)_\Phi\|_{H^{\f{3}{8},0}}.
\end{align*}
Thanks to $\pa_y^2u_2=\pa_y^3u^s \f{u(t,x,2)}{\pa_y u^s(t,2)}$, we obtain
\begin{align*}
A_4\leq C\|u_\Phi\|_{H^{\f{1}{4},1}}\|(\bw_2)_\Phi\|_{H^{\f{5}{8},0}}.
\end{align*}
Similarly, we have
\begin{align*}
A_2\leq C\|u_\Phi\|_{H^{\f{1}{4},1}}\|(\bw_2)_\Phi\|_{H^{\f{5}{8},0}}.
\end{align*}
For $A_5$, we get by Sobolev inequality that
\begin{align*}
A_5\leq& C\|u_\Phi\|_{H^{\f{1}{4},1}}\|\langle D_x\rangle^{\f12}(\bw_2)_\Phi\|_{L_y^\infty L_x^2}\leq C\|u_\Phi\|_{H^{\f{13}{4},1}}\|\pa_y(\bw_2)_\Phi\|_{H^{\f{3}{8},0}}^{\f 12}\|(\bw_2)_\Phi\|_{H^{\f{5}{8},0}}^{\f 12}\\
\leq& C(\|u_\Phi\|_{H^{\f{1}{4},1}_\mu}^2+\|(\bw_2)_\Phi\|_{H^{\f{5}{8},0}}^2)+\f{1}{16}\|\pa_y(\bw_2)_\Phi\|_{H^{{\f{3}{8}},0}}^2.
\end{align*}
It remains to estimate $A_3.$ One has
\begin{align*}
&w_2=\pa_y u^s\pa_y u-u\pa_y^2 u^s=\pa_y u^s\pa_y u_1-u_1\pa_y^2 u^s,
\end{align*}
which gives
\begin{align*}
\pa_y w_2=\pa_y u^s\pa_y^2 u_1-u_1\pa_y^3 u^s.
\end{align*}
Then we may write
\begin{align*}
A_3=&2\int_{\R^2_{+}}\psi_2(y)^2\pa_y^2 u^s\pa_y u^s\langle D_x\rangle^{\f{3}{8}}(\pa_y u_1)_\Phi \langle D_x\rangle^{\f{3}{8}}(\pa_y^2 u_1)_\Phi dxdy\\
&-2\int_{\R^2_{+}}\psi_2(y)^2\pa_y^2 u^s\pa_y^3 u^s\langle D_x\rangle^{\f{3}{8}}(\pa_y u_1)_\Phi \langle D_x\rangle^{\f{3}{8}}(u_1)_\Phi dxdy\\
=&-\int_{\R^2_{+}}\psi_2(y)^2(\pa_y^2 u^s)^2(\langle D_x\rangle^{\f{3}{8}}(\pa_y u_1)_\Phi)^2  dxdy\\
&-\int_{\R^2_{+}}\psi_2(y)^2\pa_y^3 u^s\pa_y u^s(\langle D_x\rangle^{\f{3}{8}}(\pa_y u_1)_\Phi)^2 dxdy\\
&-2\int_{\R^2_{+}}\psi_2'(y)\psi_2(y)\pa_y^2 u^s\pa_y u^s(\langle D_x\rangle^{\f{3}{8}}(\pa_y u_1)_\Phi)^2 dxdy\\
&-2\int_{\R^2_{+}}\psi_2(y)^2\pa_y u^s\pa_y^3 u^s\langle D_x\rangle^{\f{3}{8}}(\pa_y u_1)_\Phi \langle D_x\rangle^{\f{3}{8}}(\pa_yu_1)_\Phi dxdy\\
&+2\int_{\R^2_{+}}\psi_2(y)\pa_y^3 u^s\langle D_x\rangle^{\f{3}{8}}(\pa_y u_1)_\Phi \langle D_x\rangle^{\f{3}{8}}(\bw_2)_\Phi dxdy\\
\triangleq& A_{31}+\cdots+A_{35}.
\end{align*}
By Lemma \ref{lem:shear}, we have $\pa_y^2 u^s\geq \f c2>0$ on $\text{supp}\psi_2$. So,
\begin{align*}
A_{31}\leq -\f{c^2}{4}\|\psi_2(y)\langle D_x\rangle^{\f{3}{8}}(\pa_y u_1)_\Phi\|_{L^2}^2.
\end{align*}
On the other hand, $|\pa_y u^s|\le C_1\delta$ on $\text{supp}\psi_2$
with $C_1$ independent of $\delta$. So,
\begin{align*}
A_{32}+A_{34}\leq C_1\d\|\psi_2(y)\langle D_x\rangle^{\f{3}{8}}(\pa_y u_1)_\Phi\|_{L^2}^2,
\end{align*}
and
\begin{align*}
A_{35}\leq \d\|\psi_2(y)\langle D_x\rangle^{\f{3}{8}}(\pa_y u_1)_\Phi\|_{L^2}^2+C\|(\bw_2)_\Phi\|_{H^{\f{3}{8},0}}^2.
\end{align*}

While, on the support of $\psi_2'$, $\pa_yu_1$ behaves like $\bw_1$ and $\bw_2$. Similar to $A_1$, we have
\begin{align*}
A_{33}\leq C\|(\bw_1)_\Phi\|_{H^{\f12,0}}^2.
\end{align*}
This shows that
\begin{align*}
A_3\leq -\big(\f{c^2}{4}-C_1\d\big)\|\psi_2(y)\langle D_x\rangle^{\f{3}{8}}(\pa_y u_1)_\Phi\|_{L^2}^2+C\big(\|(\bw_1)_\Phi\|_{H^{\f12,0}}^2+\|(\bw_2)_\Phi\|_{H^{\f{3}{8},0}}^2\big)
\end{align*}

Putting the above estimates together and taking $\d$ small enough, we conclude our result.
\end{proof}

Next we prove Gevrey regularity estimate with weight.
We introduce a weight function $\varphi(t,y)=\varphi(y-a(t))$, where $\varphi(0)=0$ and $\varphi(y)=0$ when $|y-1|>2\d.$

\begin{proposition}\label{prop:w2-LW}
Let $w_2$ be a  solution of (\ref{eq:w2})  in $[0,T]$ and $\th_1>0$ be a small constant determined later. There exists $\delta>0$ small enough so that for any $t\in [0,T]$ and $\delta_2>0$,
\begin{align*}
&\f d{dt}\|(w_2)_\Phi \varphi^{\f{1+\th_1}{2}}\|_{H^{\f 12,0}}^2+(\lambda-C)\|(w_2)_\Phi \varphi^{\f{1+\th_1}{2}}\|_{H^{\f{3}{4}+\th,0}}^2+\|\pa_y (w_2)_\Phi \varphi^{\f{1+\th_1}{2}}\|_{H^{\f12,0}}^2\\
&\le C\Big(\big(\|u_\Phi\|_{H^{\f{1}{4},1}}^2+\|(\bw_1)_\Phi\|_{H^{\f 12,0}}^2+\|(\bw_2)_\Phi\|_{H^{\f{5}{8},0}}^2+\|(w_2)_\Phi \varphi^{\f{1+\th_1}{2}}\|_{H^{\f 12,0}}^2\Big)+\delta_2\|\pa_y(\bw_2)_\Phi\|_{H^{\f{3}{8},0}}^2.
\end{align*}
\end{proposition}

Let us  begin with the following estimate of source term.

\begin{lemma}\label{lem:weight F_2}
It holds that for any $\delta_2>0$,
\begin{align*}
\big((F_2)_\Phi,(w_2)_\Phi \varphi^{1+\th_1}\big)_{H^{\f12,0}}
\leq& C\Big(\|u_\Phi\|_{H^{\f{1}{4},1}_\mu}^2+\|(w_2)_\Phi \varphi^{\f{1+\th_1}{2}}\|_{H^{\f{3}{4},0}}^2\\
&\quad+\|(\bw_1)_\Phi\|_{H^{\f 12,0}}^2+\|(\bw_2)_\Phi\|_{H^{\f{5}{8},0}}^2\Big)+\delta_2\|\pa_y(\bw_2)_\Phi\|_{H^{\f{3}{8},0}}^2.
\end{align*}

\end{lemma}

\begin{proof}
By \eqref{eq:com1}, it is easy to show that
\begin{align*}
\big((u\pa_t\pa_y^2 u^s+[\pa_y^2 u^s,\pa_y^2]u)_\Phi,(w_2)_\Phi \varphi^{1+\th_1}\big)_{H^{\f12,0}}
\leq C\|u_\Phi\|_{H^{\f{1}{4},1}_\mu} \|(w_2)_\Phi \varphi^{\f{1+\th_1}{2}}\|_{H^{\f{3}{4},0}}.
\end{align*}
Similar to the proof of Lemma \ref{lem:F2-L}, we have
\begin{align*}
&\big((\pa_t\pa_y u^s\pa_yu+[\pa_y u^s,\pa_y^2]\pa_y u)_\Phi,(w_2)_\Phi \varphi^{1+\th_1}\big)_{H^{\f12,0}}\\
&=-2\int_{\R^2_{+}}\varphi(t,y)^{1+\th_1}\pa_y^2 u^s\langle D_x\rangle^{\f12}(\pa_y^2 u_1)_\Phi \langle D_x\rangle^{\f12}(w_2)_\Phi dxdy\\
&=-2\int_{y>a(t)}\varphi(t,y)^{1+\th_1}\pa_y^2 u^s\langle D_x\rangle^{\f12}(\pa_y^2 u_1)_\Phi \langle D_x\rangle^{\f12}(w_2)_\Phi dxdy\\
&\quad-2\int_{y<a(t)}\varphi(t,y)^{1+\th_1}\pa_y^2 u^s\langle D_x\rangle^{\f12}(\pa_y^2 u_1)_\Phi \langle D_x\rangle^{\f12}(w_2)_\Phi dxdy\\
&\quad-2\int_{y>a(t)}\varphi(t,y)^{1+\th_1}\pa_y^2 u^s\langle D_x\rangle^{\f12}(\pa_y^2 u_2)_\Phi \langle D_x\rangle^{\f12}(w_2)_\Phi dxdy\\
&=2(1+\theta_1)\int_{\R^2_{+}}\varphi^{\th_1}\pa_y \varphi\pa_y^2 u^s\langle D_x\rangle^{\f12}(\pa_y u_1)_\Phi \langle D_x\rangle^{\f12}(w_2)_\Phi dxdy\\
&\quad+2\int_{\R^2_{+}}\varphi(t,y)^{1+\th_1}\pa_y^3 u^s\langle D_x\rangle^{\f12}(\pa_y u_1)_\Phi \langle D_x\rangle^{\f12}(w_2)_\Phi dxdy\\
&\quad+2\int_{\R^2_{+}}\varphi(t,y)^{1+\th_1}\pa_y^2 u^s\langle D_x\rangle^{\f12}(\pa_y u_1)_\Phi \langle D_x\rangle^{\f12}(\pa_yw_2)_\Phi dxdy\\
&\quad-2\int_{y>a(t)}\varphi(t,y)^{1+\th_1}\pa_y^2 u^s\langle D_x\rangle^{\f12}(\pa_y^2 u_2)_\Phi \langle D_x\rangle^{\f12}(w_2)_\Phi dxdy\\&\triangleq B_1+\cdots+B_4.
\end{align*}
Here integration by parts does not give rise to the boundary term due to
$\varphi(t,a(t))=0$.

 As $\pa_y^2 u_2=\pa_y^3 u^s \f{u(t,x,2)}{\pa_y u^s(t,2)}$, we have
\begin{align*}
B_2+B_4\leq C \|u_\Phi\|_{H^{\f{1}{4},1}_\mu}\|(w_2)_\Phi \varphi^{\f{1+\th_1}{2}}\|_{H^{\f{3}{4},0}}.
\end{align*}

Thanks to $|\pa_y\varphi|\leq C$ and $\text{supp}{\pa_y \varphi}\subset[1-2\d,1+2\d]$, we get
\begin{align*}
B_1\leq& C\Big|\int_{\R}\int_{1-2\d}^{a(t)}\varphi^{\th_1}\pa_y^2 u^s\langle D_x\rangle^{\f12}(\pa_y u_1)_\Phi \langle D_x\rangle^{\f12}(w_2)_\Phi dxdy\Big|\\
&+C\Big|\int_{\R}\int_{a(t)}^{1+2\d}\varphi^{\th_1}\pa_y^2 u^s\langle D_x\rangle^{\f12}(\pa_y u_1)_\Phi \langle D_x\rangle^{\f12}(w_2)_\Phi dxdy\Big|\\
\triangleq& B_{11}+B_{12}.
\end{align*}

By Lemma \ref{lem:u-decom}, $\pa_y u_1$ can be expressed as
\begin{align}\label{eq:u1-dy}
\pa_y u_1=
\left\{
\begin{aligned}
&\pa_y^2 u^s\Big(\int_{0}^{1-2\d}\phi_1w_1dy'+\int_{1-2\d}^{y}\f{\bw_2}{(\pa_y u^s)^2}dy'\Big)+\f{\bw_2}{\pa_y u^s} \quad\text{for}\quad 1-2\d \leq y<a(t),\\
&\pa_y^2 u^s\Big(\int_{1+2\d}^y\f{\bw_2}{(\pa_y u^s)^2}dy'+\int_{2}^{1+2\d}\phi_1w_1dy'\Big)+\f{\bw_2}{\pa_y u^s}\quad\text{for}\quad a(t) < y<1+2\d.
\end{aligned}
\right.
\end{align}
Then we have
\begin{align*}
B_{11}\leq& C\Big|\int_{\R}\int_{1-2\d}^{a(t)}\varphi^{\th_1}(\pa_y^2 u^s)^2\langle D_x\rangle^{\f12}\big(\int_{0}^{1-2\d}\phi_1w_1dy'\big)_\Phi \langle D_x\rangle^{\f12}(\bw_2)_\Phi dxdy\Big|\\
&+C\Big|\int_{\R}\int_{1-2\d}^{a(t)}\varphi^{\th_1}(\pa_y^2 u^s)^2\langle D_x\rangle^{\f12}\big(\int_{1-2\d}^{y}\f{\bw_2}{(\pa_y u^s)^2}dy'\big)_\Phi \langle D_x\rangle^{\f12}(\bw_2)_\Phi dxdy\Big|\\
&+C\Big|\int_{\R}\int_{1-2\d}^{a(t)}\varphi^{\th_1}\f{\pa_y^2 u^s}{\pa_y u^s}\langle D_x\rangle^{\f12}(\bw_2)_\Phi \langle D_x\rangle^{\f12}(\bw_2)_\Phi dxdy\Big|\\
\triangleq& D_1+D_2+D_3.
\end{align*}
It is easy to get
\begin{align*}
D_1\leq C \|(\bw_1)_\Phi\|_{H^{\f12,0}} \|(\bw_2)_\Phi\|_{H^{\f{5}{8},0}}.
\end{align*}
For $y\in [1-2\d,a(t))$, $\varphi$ and $\pa_yu^s$ behaves like $|y-a(t)|$. So,
\begin{align*}
D_3\leq& C\int_{1-2\d}^{a(t)}\f{\varphi^{\th_1}}{\pa_y u^s}dy \|\langle D_x\rangle^{\f12}(\bw_2)_\Phi\|_{L_y^\infty L_x^2}^2\\
\leq& C\int_{1-2\d}^{a(t)}\f{1}{|y-a|^{1-\th_1}}dy \|\langle D_x\rangle^{\f{3}{8}}\pa_y(\bw_2)_\Phi\|_{ L^2}\|\langle D_x\rangle^{\f{5}{8}}(\bw_2)_\Phi\|_{ L^2}\\
\leq& C\|\pa_y(\bw_2)_\Phi\|_{H^{\f{3}{8},0}}\|(\bw_2)_\Phi\|_{H^{\f{5}{8},0}}.
\end{align*}
Similarly, we have
\begin{align*}
D_2\leq& C\int_{1-2\d}^{a(t)}\varphi^{\th_1}\int_{1-2\d}^{y}\f{1}{(\pa_y u^s)^2}dy'dy \|\langle D_x\rangle^{\f12}(\bw_2)_\Phi\|_{L_y^\infty L_x^2}^2\\
\leq& C\|\pa_y(\bw_2)_\Phi\|_{H^{\f{3}{8},0}}\|(\bw_2)_\Phi\|_{H^{\f{5}{8},0}}.
\end{align*}
Here we used
\begin{align*}
\int_{1-2\d}^{a(t)}\varphi^{\th_1}\int_{1-2\d}^{y}\f{1}{(\pa_y u^s)^2}dy'dy
\leq& C\int_{1-2\d}^{a(t)}\varphi^{\th_1}\int_{1-2\d}^{y}\f{1}{|y-a|^2}dy'dy\\
\leq& C\int_{1-2\d}^{a(t)}\f{1}{|y-a|^{1-\th_1}}dy\leq C.
\end{align*}
This shows that for any $\delta_2>0$,
\begin{align*}
B_{11}\leq C\big(\|(\bw_1)_\Phi\|_{H^{\f12,0}}^2+\|(\bw_2)_\Phi\|_{H^{\f{5}{8},0}}^2\big)+\delta_2\|\pa_y(\bw_2)_\Phi\|_{H^{\f{3}{8},0}}^2.
\end{align*}
The same argument shows that
\begin{align*}
B_{12}\leq C\big(\|(\bw_1)_\Phi\|_{H^{\f12,0}}^2+\|(\bw_2)_\Phi\|_{H^{\f{5}{8},0}}^2\big)+\delta_2\|\pa_y(\bw_2)_\Phi\|_{H^{\f{3}{8},0}}^2.
\end{align*}
Thus, we obtain
\begin{align*}
B_{1}\leq C\big(\|(\bw_1)_\Phi\|_{H^{\f12,0}}^2+\|(\bw_2)_\Phi\|_{H^{\f{5}{8},0}}^2\big)+2\delta_2\|\pa_y(\bw_2)_\Phi\|_{H^{\f{3}{8},0}}^2.
\end{align*}

Now we deal with $B_3.$ Similar to $A_3$ in
Lemma \ref{lem:F2-L}, we obtain
\begin{align*}
B_3=&\int_{\R^2_{+}}\varphi^{1+\th_1}(\pa_y^2 u^s)^2(\langle D_x\rangle^{\f12}(\pa_y u_1)_\Phi)^2  dxdy\\
&-\int_{\R^2_{+}}\varphi^{1+\th_1}\pa_y^3 u^s\pa_y u^s(\langle D_x\rangle^{\f12}(\pa_y u_1)_\Phi)^2 dxdy\\
&-(1+\theta_1)\int_{\R^2_{+}}\varphi^{\th_1}\pa_y \varphi\pa_y^2 u^s\pa_y u^s(\langle D_x\rangle^{\f12}(\pa_y u_1)_\Phi)^2 dxdy\\
&-2\int_{\R^2_{+}}\varphi^{1+\th_1}\pa_y u^s\pa_y^3 u^s\langle D_x\rangle^{\f12}(\pa_y u_1)_\Phi \langle D_x\rangle^{\f12}(\pa_yu_1)_\Phi dxdy\\
&+2\int_{\R^2_{+}}\varphi^{1+\th_1}\pa_y^3 u^s\langle D_x\rangle^{\f12}(\pa_y u_1)_\Phi \langle D_x\rangle^{\f12}(w_2)_\Phi dxdy\\
\triangleq& B_{31}+\cdots+B_{35}.
\end{align*}
Similar to $A_{31},A_{32},A_{34},A_{35}$ in Lemma \ref{lem:F2-L}, we have
\begin{align*}
B_{31}+B_{32}+B_{34}+B_{35}\leq -\big(\f{c^2}{4}-C_1\d\big)\|(\pa_y u_1)_\Phi \varphi^{\f{1+\th_1}{2}}\|_{H^{\f 12,0}}^2+C\|(w_2)_\Phi\varphi^{\f{1+\th_1}{2}}\|_{H^{\f 12,0}}^2.
\end{align*}

Similar to $B_{1}$, we have
\begin{align*}
B_{33}\leq& C\Big|\int_{\R}\int_{1-2\d}^{a(t)}\varphi^{\th_1}\pa_y u^s(\langle D_x\rangle^{\f12}(\pa_y u_1)_\Phi)^2  dxdy\Big|\\
&+C\Big|\int_{\R}\int_{a(t)}^{1+2\d}\varphi^{\th_1}\pa_y u^s(\langle D_x\rangle^{\f12}(\pa_y u_1)_\Phi)^2  dxdy\Big|
\triangleq E_{1}+E_{2}.
\end{align*}
We get by (\ref{eq:u1-dy})  that
\begin{align*}
E_1\leq& C\Big|\int_{\R}\int_{1-2\d}^{a(t)}\varphi^{\th_1}\pa_y u^s\big(\int_{0}^{1-2\d}\langle D_x\rangle^{\f12}(\phi_1w_1)_\Phi dy'\big)^2  dxdy\Big|\\
&+C\Big|\int_{\R}\int_{1-2\d}^{a(t)}\varphi^{\th_1}\pa_y u^s\big(\int_{1-2\d}^{y}\f{\langle D_x\rangle^{\f12}(\bw_2)_\Phi}{(\pa_y u^s)^2}dy'\big)^2 dxdy\Big|\\
&+C\Big|\int_{\R}\int_{1-2\d}^{a(t)}\varphi^{\th_1}\f{1}{\pa_y u^s}(\langle D_x\rangle^{\f12}(\bw_2)_\Phi)^2dxdy\Big|\\
\triangleq& E_{11}+E_{12}+E_{13}.
\end{align*}
It is easy to see that
\begin{align*}
E_{11}\leq C\|(\bw_1)_\Phi\|_{H^{\f 12,0}}^2,
\end{align*}
and
\begin{align*}
E_{13} \leq& C\int_{1-2\d}^{a}\f{\varphi^{\th_1}}{\pa_y u^s}dy \|\langle D_x\rangle^{\f12}(\bw_2)_\Phi\|_{L_y^\infty L_x^2}^2\leq C\|\pa_y(\bw_2)_\Phi\|_{H^{\f{3}{8},0}}\|(\bw_2)_\Phi\|_{H^{\f{5}{8},0}}.
\end{align*}
Similar to $E_{13}$, we have
\begin{align*}
E_{12}\leq& C\|\pa_y(\bw_2)_\Phi\|_{H^{\f{3}{8},0}}\|(\bw_2)_\Phi\|_{H^{\f{5}{8},0}}.
\end{align*}
This shows that for any $\delta_1>0$,
\begin{align*}
E_{1}\leq C\big(\|(\bw_1)_\Phi\|_{H^{\f 12,0}}^2+\|(\bw_2)_\Phi\|_{H^{\f{5}{8},0}}^2\big)+\delta_2\|\pa_y(\bw_2)_\Phi\|_{H^{\f{3}{8},0}}^2.
\end{align*}
Similarly, we have
\begin{align*}
E_{2}\leq C\big(\|(\bw_1)_\Phi\|_{H^{\f 12,0}}^2+\|(\bw_2)_\Phi\|_{H^{\f{5}{8},0}}^2\big)+\delta_2\|\pa_y(\bw_2)_\Phi\|_{H^{\f{3}{8},0}}^2.
\end{align*}
Thus, we get
\begin{align*}
B_{33}\leq C\big(\|(\bw_1)_\Phi\|_{H^{\f 12,0}}^2+\|(\bw_2)_\Phi\|_{H^{\f{5}{8},0}}^2\big)+2\delta_2\|\pa_y(\bw_2)_\Phi\|_{H^{\f{3}{8},0}}^2.
\end{align*}
Summing up, we obtain
\begin{align*}
B_3\leq& -\big(\f{c^2}{4}-C_1\d\big)\|(\pa_y u_1)_\Phi \varphi^{\f{1+\th_1}{2}}\|_{H^{\f 12,0}}^2+C\big(\|(w_2)_\Phi\varphi^{\f{1+\th_1}{2}}\|_{H^{\f 12,0}}^2+\|(\bw_1)_\Phi\|_{H^{\f 12,0}}^2\\
&+\|(\bw_2)_\Phi\|_{H^{\f{5}{8},0}}^2\big)+2\delta_2\|\pa_y(\bw_2)_\Phi\|_{H^{\f{3}{8},0}}^2.
\end{align*}

Summing up the estimates of $B_1,B_2, B_3$ and
taking $\delta$ small enough, we conclude the lemma.
\end{proof}

\smallskip

Now we are in position to prove Proposition \ref{prop:w2-LW}.

\begin{proof}
Recall that $w_2$ satisfies
\begin{align}
&\pa_tw_2+u^s\pa_xw_2-\pa_y^2w_2=F_2.\label{eq: linearized w2}
\end{align}
Applying $e^{\Phi(t,D_x)}$ to (\ref{eq: linearized w2}), we get
\begin{align}
\pa_t (w_2)_\Phi+\la\langle D_x \rangle ^{\f12+2\th} (w_2)_\Phi+u^s\pa_x(w_2)_\Phi-\pa_y^2(w_2)_\Phi=(F_2)_\Phi.\label{eq: linearized w2 phi}
\end{align}
Taking $\langle D_x \rangle ^{\f72}$ on both sides of (\ref{eq: linearized w2 phi}) and taking $L^2$ inner product with $\langle D_x \rangle ^{\f12}(w_2)_\Phi \varphi^{1+\th_1}$, we obatin
\begin{align*}
\f12\f{d}{dt}&\|(w_2)_\Phi \varphi^{\f{1+\th_1}{2}}\|_{H^{\f 12,0}}^2
+\la\| (w_2)_\Phi\varphi^{\f{1+\th_1}{2}}\|_{H^{\f {3}{4}+\th,0}}^2
+\int_{\R^2_{+}}u^s\langle D_x \rangle ^{\f12}\pa_x(w_2)_\Phi\langle D_x \rangle ^{\f12}(w_2)_\Phi\varphi^{1+\th_1}dxdy\\
&\qquad-\int_{\R^2_{+}}\langle D_x \rangle ^{\f12}\pa_y^2(w_2)_\Phi\langle D_x \rangle ^{\f12}(w_2)_\Phi\varphi^{1+\th_1}dxdy\\
\leq&\int_{\R^2_{+}}\varphi^{\th_1}\pa_t \varphi|\langle D_x \rangle ^{\f12}(w_2)_\Phi|^2dxdy+\big((F_2)_\Phi,(w_2)_\Phi \varphi^{1+\th_1}\big)_{H^{\f12,0}}.
\end{align*}
We get by integration by parts that
\begin{align*}
\int_{\R^2_{+}}u^s\langle D_x \rangle ^{\f12}\pa_x(w_2)_\Phi\langle D_x \rangle ^{\f12}(w_2)_\Phi\varphi^{1+\th_1}dxdy=0,
\end{align*}
and
\begin{align*}
&-\int_{\R^2_{+}}\langle D_x \rangle ^{\f12}\pa_y^2(w_2)_\Phi\langle D_x \rangle ^{\f12}(w_2)_\Phi\varphi^{1+\th_1}dxdy\\
&= \|\pa_y(w_2)_\Phi \varphi^{\f{1+\th_1}{2}}\|_{H^{\f 12,0}}^2
+(1+\th_1)\int_{\R^2_{+}}\varphi^{\th_1}\pa_y\varphi \langle D_x \rangle ^{\f12}\pa_y(w_2)_\Phi\langle D_x \rangle ^{\f12}(w_2)_\Phi dxdy\\
&\geq \|\pa_y(w_2)_\Phi \varphi^{\f{1+\th_1}{2}}\|_{H^{\f 12,0}}^2-C\|\pa_y(\bw_2)_\Phi \|_{H^{\f{3}{8},0}}\|(\bw_2)_\Phi \|_{H^{\f{5}{8},0}}\\
&\geq \|\pa_y(w_2)_\Phi \varphi^{\f{1+\th_1}{2}}\|_{H^{\f 12,0}}^2-C\|(\bw_2)_\Phi \|_{H^{\f{5}{8},0}}^2-\delta_2\|\pa_y(\bw_2)_\Phi \|_{H^{\f{3}{8},0}}^2.
\end{align*}
Here we used $|\pa_y \varphi|\leq C$ and $\psi_2(y)=1$ for $y\in  \text{supp}\pa_y \varphi$. Similarly, we have
\begin{align*}
\int_{\R^2_{+}}\varphi^{\th_1}\pa_t \varphi|\langle D_x \rangle ^{\f12}(w_2)_\Phi|^2dxdy\leq C\|(\bw_2)_\Phi\|_{H^{\f 12,0}}^2.
\end{align*}

With the above estimates and using Lemma \ref{lem:weight F_2},  we deduce our result.
\end{proof}

\section{Gevrey regularity estimate of $h$}

\begin{proposition}\label{prop:linearized h}
Let $h$ be a  solution of (\ref{eq:h}) in $[0,T]$. Then it holds that
\begin{align*}
&\f d{dt}\|h_{\Phi}\|_{L^2}^2+\lambda\|h_{\Phi}\|_{H^{\f{1}{4}+\th,0}}^2+\|\pa_yh_{\Phi}\|_{L^{2}}^2\\
&\le C\Big(\|h_{\Phi}\|_{L^2}^2+\|u_{\Phi}\|_{H^{\f{1}{4},1}}^2+\|(\bw_1)_\Phi\|_{H^{\f{3}{4},0}}^2+\|(\bw_2)_\Phi\|_{H^{\f{5}{8},0}}^2+\|(w_2)_\Phi\varphi^{\f{1+\th_1}{2}}\|_{H^{\f{3}{4},0}}^2\Big).
\end{align*}
\end{proposition}

The proposition follows from the following two lemmas.

\begin{lemma}\label{lem: 1h}
Let $h$ be a smooth solution of (\ref{eq:h}) in $[0,T]$. Then it holds that
\begin{align*}
&\f d{dt}\|h_{\Phi}\|_{L^{2}}^2+\lambda\|h_{\Phi}\|_{H^{\f{1}{4}+\th,0}}^2+\|\pa_yh_{\Phi}\|_{L^{2}}^2\\
&\le C\big(\|h_{\Phi}\|_{L^{2}}^2+\|u_{\Phi}\|_{H^{\f{1}{4},1}}^2\big)+2\big(\phi_3'(y)v_\Phi,\phi_3(y)u_\Phi \big)_{L^2}.
\end{align*}
\end{lemma}

\begin{proof}
Applying $e^{\Phi(t,D_x)}$ on (\ref{eq:h}) and making $H^{3,0}$ energy estimate, we obtain
\begin{align*}
\f d{dt}&\|h_{\Phi}\|_{L^{2}}^2+\lambda\|h_{\Phi}\|_{H^{\f{1}{4}+\th,0}}^2-\big( (\pa_y^2h)_\Phi,h_\Phi \big)_{L^{2}}\\
\leq&-\big(u^s\pa_x h_\Phi,h_\Phi \big)_{L^{2}}+\big((\pa_t d-\pa_y^2d)(\pa_y u)_\Phi,h_\Phi \big)_{L^{2}}-2\big((\pa_yd\pa_y^2 u)_\Phi,h_\Phi \big)_{L^{2}}\\
&-\big(d(v\pa_y^2 u^s)_\Phi,h_\Phi \big)_{L^{2}}.
\end{align*}
Thanks to $h_\Phi|_{y=0}=0,$ we get by integration by parts that
\begin{align*}
&\big(u^s\pa_x h_\Phi,h_\Phi \big)_{L^{2}}=0,\quad -\big( (\pa_y^2h)_\Phi,h_\Phi \big)_{L^{2}}=\|(\pa_yh)_{\Phi}\|_{L^{2}}^2,
\end{align*}
and
\begin{align*}
-2\big((\pa_yd\pa_y^2 u)_\Phi,h_\Phi \big)_{L^{2}}=&2\big((\pa_y^2d\pa_y u)_\Phi,h_\Phi \big)_{L^2}+2\big((\pa_yd\pa_y u)_\Phi,\pa_yh_\Phi \big)_{L^{2}}\\
\leq&
C\big(\|u_{\Phi}\|_{H^{\f{1}{4},1}}^2+\|h_{\Phi}\|_{L^{2}}^2\big)+\f{1}{16}\|(\pa_yh)_{\Phi}\|_{L^{2}}^2.
\end{align*}
After some calculations, we have
\begin{align*}
\pa_t d-\pa_y^2d=\f{\pa_y^3 u^s\phi_3'}{(\pa_y^2 u^s)^{\f 32}}-\f{\phi_3''}{(\pa_y^2 u^s)^{\f 12}}-\f 34\f{(\pa_y^3 u^s)^2\phi_3}{(\pa_y^2 u^s)^{\f 52}},
\end{align*}
which gives
\begin{align*}
\big((\pa_t d-\pa_y^2d)(\pa_y u)_\Phi,h_\Phi \big)_{L^{2}}\leq&
C\|u_{\Phi}\|_{H^{\f{1}{4},1}}\|h_{\Phi}\|_{L^{2}}.
\end{align*}
Using $\pa_xu+\pa_yv=0$, we get by integration by parts that
\begin{align*}
-\big(d(v\pa_y^2 u^s)_\Phi,h_\Phi \big)_{L^{2}}=&-\big(\phi_3(y)(\pa_y^2u^s)^{-1/2}\pa_y^2 u^s v_\Phi,\phi_3(y)(\pa_y^2u^s)^{-1/2}(\pa_y u)_\Phi \big)_{L^{2}}\\
=&-\big(\phi_3(y) v_\Phi,\phi_3(y)(\pa_y u)_\Phi \big)_{L^{2}}\\
=&2\big(\phi_3'(y)v_\Phi,\phi_3(y)u_\Phi \big)_{L^2}+\big(\phi_3(y)(\pa_yv)_\Phi,\phi_3(y)u_\Phi \big)_{L^2}\\
=&2\big(\phi_3'(y)v_\Phi,\phi_3(y)u_\Phi \big)_{L^2}.
\end{align*}
This completes the proof of the lemma.
\end{proof}

The following lemma is devoted to the most trouble term $\big(\phi_3'(y) v_\Phi,\phi_3(y)u_\Phi \big)_{L^2}$. The argument is motivated by \cite{GM}.

\begin{lemma}\label{lem:2h}
It holds that
\begin{align*}
\big(\phi_3'(y)v_\Phi,\phi_3(y)u_\Phi \big)_{L^2}\leq C\Big(\|u_\Phi\|_{H^{\f{1}{4}+\th,1}}^2+\|(\bw_1)_\Phi\|_{H^{\f{3}{4},0}}^2+\|(\bw_2)_\Phi\|_{H^{\f{5}{8},0}}^2+\|(w_2)_\Phi\varphi^{\f{1+\th_1}{2}}\|_{H^{\f{3}{4},0}}^2\Big).
\end{align*}

\end{lemma}

\begin{proof}
Let $\text{supp}\phi_3'=E_1\cup E_2,$ where $E_1=[\f 12,\f 34]$ and $E_2=[\f74, 2]$. Then we write
\begin{align*}
\Big|\int_{\R_{+}^2}\phi_3\phi_3'v_\Phi u_\Phi dxdy\Big|
\leq&
\Big|\int_{\R}\int_{E_1}\phi_3\phi_3'\int_{0}^{y}\pa_xu_\Phi dy'u_\Phi dxdy\Big|\\
&+
\Big|\int_{\R}\int_{E_2}\phi_3\phi_3'\int_{0}^{y}\pa_xu_\Phi dy'u_\Phi dxdy\Big|\\
\triangleq& J_1+J_2.
\end{align*}

In $E_1$,  $u$ can be expressed as $u=\pa_y u^s\int_{0}^{y}\bw_1dy'$ so that
\begin{align*}
J_1\leq& \Big|\int_{\R}\int_{E_1}\phi_3\phi_3'\int_{0}^{y}\pa_xu_\Phi dy'~\pa_y u^s\int_{0}^{y}(\bw_1)_\Phi dy' dxdy\Big|\\
\leq&C\|(\bw_1)_\Phi\|_{H^{\f{3}{4},0}}\|u_\Phi\|_{H^{\f{1}{4},0}}
\end{align*}
In $E_2$, $u$ can be expressed as $u=\pa_y u^s\int_{2}^{y}\bw_1dy'+\pa_y u^s\f{u(t,x,2)}{\pa_y u^s(t,2)}$ so that
\begin{align*}
J_2\leq&\Big|\int_{\R}\int_{E_2}\phi_3\phi_3'\int_{0}^{y}\pa_xu_\Phi dy'\big(\pa_y u^s\int_{2}^{y}\bw_1dy'\big)_\Phi dxdy\Big|\\
&+\Big|\int_{\R}\int_{E_2}\phi_3\phi_3'\int_{0}^{y}\pa_xu_\Phi dy'\pa_y u^s\f{(u(t,x,2))_\Phi}{\pa_y u^s(t,2)} dxdy\Big|\\
\triangleq& J_{21}+J_{22}.
\end{align*}
Similar to $J_1$, we have
\begin{align*}
J_{21}
\leq&C\|(\bw_1)_\Phi\|_{H^{\f{3}{4},0}}\|u_\Phi\|_{H^{\f{1}{4},0}}.
\end{align*}
Recall that $a(t)$ is a critical point of $u^s$. We decompose $\int_{0}^{y}\pa_xu_\Phi dy'$ into the following three parts
\begin{align*}
\int_{0}^{y}\pa_xu_\Phi dy'=\int_{0}^{a(t)}\pa_xu_\Phi dy'+\int_{a(t)}^{\f 74}\pa_xu_\Phi dy'+\int_{\f 74}^{y}\pa_xu_\Phi dy'.
\end{align*}
Then we have
\begin{align*}
J_{22}\leq&
\Big|\int_{\R}\int_{E_2}\phi_3\phi_3'\int_{0}^{a(t)}\pa_xu_\Phi dy'~\pa_y u^s\f{(u(t,x,2))_\Phi}{\pa_y u^s(t,2)} dxdy\Big|\\
&+\Big|\int_{\R}\int_{E_2}\phi_3\phi_3'\int_{a(t)}^{\f74}\pa_xu_\Phi dy'\pa_y u^s\f{(u(t,x,2))_\Phi}{\pa_y u^s(t,2)} dxdy\Big|\\
&+\Big|\int_{\R}\int_{E_2}\phi_3\phi_3'\int_{\f74}^{y}\pa_xu_\Phi dy'\pa_y u^s\f{(u(t,x,2))_\Phi}{\pa_y u^s(t,2)} dxdy\Big|\\
\triangleq& K_{1}+K_{2}+K_{3.}
\end{align*}
By Lemma \ref{lem:u-decom}, we get
\begin{align*}
K_1\leq&
\Big|\int_{\R}\int_{E_2}\phi_3\phi_3'\int_{0}^{1-2\d}\pa_xu_\Phi dy'~\pa_y u^s\f{(u(t,x,2))_\Phi}{\pa_y u^s(t,2)} dxdy\Big|\\
&+\Big|\int_{\R}\int_{E_2}\phi_3\phi_3'\int_{1-2\d}^{a(t)}\pa_x\big(\pa_y u^s\int_{0}^{1-2\d}\phi_1w_1dy''\big)_\Phi dy'~\pa_y u^s\f{(u(t,x,2))_\Phi}{\pa_y u^s(t,2)} dxdy\Big|\\
&+\Big|\int_{\R}\int_{E_2}\phi_3\phi_3'\int_{1-2\d}^{a(t)}\pa_x(\pa_y u^s\int_{1-2\d}^{y'}\f{\bw_2}{(\pa_y u^s)^2}dy'')_\Phi dy'~\pa_y u^s\f{(u(t,x,2))_\Phi}{\pa_y u^s(t,2)} dxdy\Big|\\
\triangleq& K_{11}+K_{12}+K_{13}.
\end{align*}
As in $J_1,$ we have
\begin{align*}
K_{11}+K_{12}\leq C\|(\bw_1)_\Phi\|_{H^{\f{3}{4},0}}\|u_\Phi\|_{H^{\f{1}{4},1}}.
\end{align*}
For $K_{13}$, let us first estimate
\begin{align*}
\Big\|&\int_{1-2\d}^{a(t)}\D^{\f{3}{4}-\th}(\pa_y u^s\int_{1-2\d}^{y'}\f{\bw_2}{(\pa_y u^s)^2}dy'')_\Phi dy'\Big\|_{L_x^2}\\\leq&
\int_{1-2\d}^{a(t)}\pa_y u^s\int_{1-2\d}^{y'}\f{\|\D^{\f{3}{4}-\th}(\bw_2)_\Phi\|_{L_x^2}}{(\pa_y u^s)^2}dy''dy'\\
\leq&
\int_{1-2\d}^{a(t)}\pa_y u^s\int_{1-2\d}^{y'}\f{\|\D^{\f{5}{8}}(\bw_2)_\Phi\|^{1-\al}_{L_x^2}\|\D^{\f{3}{4}}(w_2)_\Phi\|^{\al}_{L_x^2}\varphi^{\f{\al(1+\th_1)}{2}}}{(\pa_y u^s)^2\varphi^{\f{\al(1+\th_1)}{2}}} dy''dy'\\
\leq&
\int_{1-2\d}^{a}\pa_y u^s\Big(\int_{1-2\d}^{y'}\f{1}{(\pa_y u^s)^4\varphi^{\al(1+\th_1)}} dy''\Big)^{\f12}dy'\|\D^{\f{5}{8}}(\bw_2)_\Phi\|^{1-\al}_{L^2}\|\D^{\f{3}{4}}(w_2)_\Phi\varphi^{\f{1+\th_1}{2}}\|_{L^2}^\al \\
\leq&
\int_{1-2\d}^{a}|y'-a(t)|^{-\f12-\f{\al(1+\th_1)}{2}}dy'   \|\D^{\f{5}{8}}(\bw_2)_\Phi\|^{1-\al}_{L^2}\|\D^{\f{3}{4}}(w_2)_\Phi\varphi^{\f{1+\th_1}{2}}\|_{L^2}^\al\\
\leq&
C\|\D^{\f{5}{8}}(\bw_2)_\Phi\|^{1-\al}_{L^2}\|\D^{\f{3}{4}}(w_2)_\Phi\varphi^{\f{1+\th_1}{2}}\|_{L^2}^\al.
\end{align*}
Here $\al=1-8\th$ and take $\f{\al(1+\th_1)}{2}<\f12$ to ensure that $\int_{1-2\d}^{a}|y'-a(t)|^{-\f12-\f{\al(1+\th_1)}{2}}dy'\leq C.$ As a result, we obtain
\begin{align*}
K_{13}\leq C\big(\|u_\Phi\|_{H^{\f{1}{4},1}}^2+\|(\bw_2)_\Phi\|_{H^{\f{5}{8},0}}^2+\|(w_2)_\Phi\varphi^{\f{1+\th_1}{2}}\|_{H^{\f{3}{4},0}}^2\big).
\end{align*}
Then we have
\begin{align*}
K_{1}\leq C\big(\|u_\Phi\|_{H^{\f{1}{4},1}}^2+\|(\bw_1)_\Phi\|_{H^{\f{3}{4},0}}^2+\|(\bw_2)_\Phi\|_{H^{\f{5}{8},0}}^2+\|(w_2)_\Phi\varphi^{\f{1+\th_1}{2}}\|_{H^{\f{3}{4},0}}^2\big).
\end{align*}
For $K_2$, by Lemma \ref{lem:u-decom}, we  have
\begin{align*}
K_2\leq&\Big|\int_{\R}\int_{E_2}\phi_3\phi_3'\int_{a(t)}^{1+2\d}\pa_x\big(\pa_y u^s\int_{1+2\d}^{y'}\f{\bw_2}{(\pa_y u^s)^2}dy''\big)_\Phi dy'\pa_y u^s\f{(u(t,x,2))_\Phi}{\pa_y u^s(t,2)} dxdy\Big|\\
&+\Big|\int_{\R}\int_{E_2}\phi_3\phi_3'\int_{a(t)}^{1+2\d}\pa_x\big(\pa_y u^s\int_{2}^{1+2\d}\bw_1dy''\big)_\Phi dy'\pa_y u^s\f{(u(t,x,2))_\Phi}{\pa_y u^s(t,2)} dxdy\Big|\\
&+\Big|\int_{\R}\int_{E_2}\phi_3\phi_3'\int_{a(t)}^{1+2\d}\pa_y u^s\f{(\pa_xu(t,x,2))_\Phi}{\pa_y u^s(t,2)} dy'\pa_y u^s\f{(u(t,x,2))_\Phi}{\pa_y u^s(t,2)} dxdy\Big|\\
&+\Big|\int_{\R}\int_{E_2}\phi_3\phi_3'\int_{1+2\d}^{\f74}\pa_x\big(\pa_y u^s\int_{2}^{y'}\bw_1dy'')_\Phi dy'\pa_y u^s\f{(u(t,x,2)\big)_\Phi}{\pa_y u^s(t,2)} dxdy\Big|\\
&+\Big|\int_{\R}\int_{E_2}\phi_3\phi_3'\int_{1+2\d}^{\f74}\pa_y u^s\f{(\pa_xu(t,x,2))_\Phi}{\pa_y u^s(t,2)} dy'~\pa_y u^s\f{(u(t,x,2))_\Phi}{\pa_y u^s(t,2)} dxdy\Big|\\
\triangleq&K_{21}+\cdots+K_{25}.
\end{align*}
Similar to $K_{13}$, we have
\begin{align*}
K_{21}\leq C\big(\|u_\Phi\|_{H^{\f{1}{4},1}}^2+\|(\bw_2)_\Phi\|_{H^{\f{5}{8},0}}^2+\|(w_2)_\Phi\varphi^{\f{1+\th_1}{2}}\|_{H^{\f{3}{4},0}}^2\big).
\end{align*}
Similar to $K_{12}$, we have
\begin{align*}
K_{22}+K_{24}\leq C\|(\bw_1)_\Phi\|_{H^{\f{3}{4},0}}\|u_\Phi\|_{H^{\f{1}{4},1}}.
\end{align*}
Integration by parts, we have
\begin{align*}
K_{23}+K_{25}=0.
\end{align*}
This shows that
\begin{align*}
K_{2}\leq C\big(\|u_\Phi\|_{H^{\f{1}{4},1}}^2+\|(\bw_1)_\Phi\|_{H^{\f{3}{4},0}}^2+\|(\bw_2)_\Phi\|_{H^{\f{5}{8},0}}^2+\|(w_2)_\Phi\varphi^{\f{1+\th_1}{2}}\|_{H^{\f{3}{4},0}}^2\big).
\end{align*}
Similar to $K_{24}$ and $K_{25}$, we have
\begin{align*}
K_{3}\leq C\|(\bw_1)_\Phi\|_{H^{\f{3}{4},0}}\|u_\Phi\|_{H^{\f{1}{4},1}}.
\end{align*}
Summing up the estimates of $K_1, K_2, K_3$, we deduce that
\begin{align*}
J_{22}\leq C\big(\|u_\Phi\|_{H^{\f{1}{4},1}}^2+\|(\bw_1)_\Phi\|_{H^{\f{3}{4},0}}^2+\|(\bw_2)_\Phi\|_{H^{\f{5}{8},0}}^2+\|(w_2)_\Phi\varphi^{\f{1+\th_1}{2}}\|_{H^{\f{3}{4},0}}^2\big).
\end{align*}

Putting the above estimates together, we conclude the lemma.
\end{proof}

\section{Proof of Theorem \ref{linearized main thm}}

Let us first recover the regularity of $u$ from $w_1$ and $h$.

\begin{lemma}\label{lem:relation linearized uwh}
It holds that
\begin{align*}
\|u_\Phi\|_{H^{\f{1}{4}+\th,1}_\mu}\leq&
C\big(\|(\bw_1)_\Phi\|_{H^{\f{1}{2},0}}+\|h_\Phi\|_{H^{\f{1}{4}+\th,0}}\big).\end{align*}
\end{lemma}

\begin{proof}

The proof is split into two steps.

{\bf Step 1.}\, Estimate of $\|\pa_yu_\Phi\|_{H^{\f{1}{4}+\th,0}_\mu}$

First of all, we have
\begin{align*}
\|\pa_yu_\Phi\|_{H^{\f{1}{4}+\th,0}_\mu}\leq&
\|1_{[0,\f 34]}(y)\pa_yu_\Phi\|_{H^{\f{1}{4}+\th,0}_\mu}+\|1_{[\f 34,\f 74]}(y)\pa_yu_\Phi\|_{H^{\f{1}{4}+\th,0}_\mu}+\|1_{[\f 74,+\infty)}(y)\pa_yu_\Phi\|_{H^{\f{1}{4}+\th,0}_\mu}\\
\triangleq& I_1+I_2+I_3.
\end{align*}
For $y\in [0,\f 34]\cup [\f 74,+\infty)$, we have
\beno
\pa_yu(y)=\pa_y u^s\big(w_1-\pa_y(\f{1}{\pa_y u^s})u\big),
\eeno
which gives
\begin{align*}
I_1+I_3
\leq C\|(\bw_1)_\Phi\|_{H^{\f{1}{2},0}}+C\|u_\Phi\|_{H^{\f{1}{4}+\th,0}_\mu}.
\end{align*}
For $y\in [\f 34,\f 74]$, using $\pa_yu={h}(\pa_y^2 u^s)^{-\f 12}$, we get
\begin{align*}
I_{2}\leq&C\|h_\Phi\|_{H^{\f{1}{4}+\th,0}}.
\end{align*}

{\bf Step 2.} Estimate of $\|u_\Phi\|_{H^{\f{1}{4}+\th,0}_\mu}$

We have
\begin{align*}
\|u_\Phi\|_{H^{\f{1}{4}+\th,0}_\mu}\leq&
\|1_{[0,\f 34]}(y)u_\Phi\|_{H^{\f{1}{4}+\th,0}_\mu}+\|1_{[\f 34,\f 74]}(y)u_\Phi\|_{H^{\f{1}{4}+\th,0}_\mu}+\|1_{[\f 74,+\infty)}(y)u_\Phi\|_{H^{\f{1}{4}+\th,0}_\mu}\\
\triangleq& I_4+I_5+I_6.
\end{align*}
For $y\in [0,\f 34]$, we have
\beno
u(y)=\pa_y u^s\Big(\int_{0}^{y}w_1dy'\Big),
\eeno
which gives
\begin{align*}
I_4\le C\|(\bw_1)_\Phi\|_{H^{\f{1}{2},0}}.
\end{align*}
For $y\in [\f 34,\f 74]$, we have
\ben
\tu(y)=\int_{\f 34}^{y}\f{h}{(\pa_y^2 u^s)^{\f 12}}dy'+u(t,x,\f 34),
\een
from which and the estimate of $I_4$, we deduce that
\begin{align*}
I_5\leq
C\big(\|(\bw_1)_\Phi\|_{H^{\f{1}{2},0}}+\|h_\Phi\|_{H^{\f{1}{4}+\th,0}}\big).
\end{align*}
For $y\in [\f 74,+\infty)$, we have
\beno
\tu(y)=\pa_y u^s \Big(\int_{\f 74}^{y}w_1dy'\Big)+\pa_y u^s\f{u(t,x,\f 74)}{\pa_y u^s(t,\f 74)}
\eeno
from which and the estimate of $I_5$, we deduce that
\begin{align*}
I_{6}\leq&
C\big(\|(\bw_1)_\Phi\|_{H^{\f{1}{2},0}}+\|h_\Phi\|_{H^{\f{1}{4}+\th,0}}\big).
\end{align*}

Now, the inequality follows by putting the estimates of $I_1-I_6$ together.
\end{proof}
\medskip

Now we are in position to prove Theorem \ref{linearized main thm}.

\begin{proof}
The approximate solution can be easily constructed by adding the viscous term $-\epsilon^2\pa_x^2u$ to \eqref{eq:pran-L}. So, we just present the uniform estimate. For this end, we introduce
\begin{align*}
\mathcal{E}(t)\triangleq& \|(\bw_1)_\Phi\|_{H^{\f 12,0}}^2+\|(\bw_2)_\Phi\|_{H^{\f{3}{8},0}}^2+\|(w_2)_\Phi\varphi^{\f{1+\th_1}{2}}\|_{H^{\f{1}{2},0}}^2+\|h_\Phi\|_{L^{2}}^2,\\
\mathcal{D}(t)\triangleq& \|\pa_y(\bw_1)_\Phi\|_{H^{\f 12,0}}^2+\|\pa_y(\bw_2)_\Phi\|_{H^{\f{3}{8},0}}^2+\|\pa_y(w_2)_\Phi\varphi^{\f{1+\th_1}{2}}\|_{H^{\f{1}{2},0}}^2+\|\pa_yh_\Phi\|_{L^{2}}^2,\\
\mathcal{G}(t)\triangleq& \|(\bw_1)_\Phi\|_{H^{\f {3}{4}+\th,0}}^2+\|(\bw_2)_\Phi\|_{H^{\f{5}{8}+\th,0}}^2+\|(w_2)_\Phi\varphi^{\f{1+\th_1}{2}}\|_{H^{\f{3}{4}+\th,0}}^2+\|h_\Phi\|_{H^{\f{1}{4}+\th,0}}^2.
\end{align*}

Choosing $\la$ large enough and $\delta_2$ suitably small, we infer from Proposition
\ref{prop:w1-L}, Proposition \ref{prop:w2}, Proposition \ref{prop:w2-LW}, Proposition \ref{prop:linearized h} and Lemma \ref{lem:relation linearized uwh} that
\begin{align*}
\f{d}{dt}\mathcal{E}(t)+\lambda\mathcal{G}(t)+ \mathcal{D}(t)
\leq&
C\mathcal{E}(t).
\end{align*}
Then Gronwall's inequality gives
\begin{align}\label{eq:energy inequ-L}
\mathcal{E}(t)+\lambda\int_0^t\mathcal{G}(s)ds+\int_{0}^{t}\mathcal{D}(s)ds\leq \mathcal{E}(0)e^{Ct}
\end{align}
for any $t\in [0,T]$.
\end{proof}

\section{Note on well-posedness in Gevrey class 2}

Let us explain how to use a new unknown $h_1=\pa_y^2u-\f {\pa_y^3u^s} {\pa_y^2u^s}\pa_yu$ 
introduced in \cite{LY} to obtain the well-posedness of \eqref{eq:pran-L}  in Gevrey class 2 in our framework.
It is easy to verify that $h_1$ satisfies the following equation
\beno
\pa_th_1+u^s\pa_xh_1+\pa_xw_2-\pa^2_yh_1=\pa_t\Big(\f {\pa_y^3u^s} {\pa_y^2u^s}\Big)\pa_yu+\Big[\f {\pa_y^3u^s} {\pa_y^2u^s},\pa_y^2\Big]\pa_yu.
\eeno
The unknown $h_1$ is well-defined in non-monotonic domain. It is easy to show that 
\beno
(\bh_1)_\Phi\in L^\infty(0,T;L^2)\cap L^2(0,T;H^{\f14,0}),\quad \bh_1=\phi_3(y)h_1,
\eeno      
if $(\bw_2)_\Phi\in L^2(0,T;H^{\f34,0})$. On the other hand, if we know that $(\bh_1)_\Phi\in L^2(0,T;H^{\f14,0})$ which will imply 
$\pa_y^2u_\Phi\in L^2(0,T;H^{\f14,0})$ because of $\pa_yu_\Phi\in L^2(0,T;H^{\f14,0})$ by Lemma \ref{lem:relation linearized uwh},
we can show that $(\bw_2)_\Phi\in L^2(0,T;H^{\f34,0})$ by following the proof of Proposition \ref{prop:w2}.
More precisely, we can deduce that
\begin{align*}
&\f d{dt}\|(\bw_1)_\Phi\|_{H^{\f{1}{2},0}}^2+(\lambda-C)\|(\bw_1)_\Phi\|_{H^{\f{3}{4},0}}^2+\|\pa_y (\bw_1)_\Phi\|_{H^{\f{1}{2},0}}^2\\
&\quad\le C\Big(\|u_\Phi\|_{H^{\f{1}{4},1}_\mu}^2+\|(\bw_1)_\Phi\|_{H^{\f{1}{2},0}}^2+\|(\bw_2)_\Phi\|_{H^{\f{1}{2},0}}^2\Big),\\
&\f d{dt}\|(\bw_2)_\Phi\|_{H^{\f{1}{2},0}}^2+(\lambda-C)\|(\bw_2)_\Phi\|_{H^{\f{3}{4},0}}^2+\|\pa_y (\bw_2)_\Phi\|_{H^{\f{1}{2},0}}^2\\
&\quad\le C\Big(\|u_\Phi\|_{H^{\f{1}{4},1}}^2+\|(\bw_1)_\Phi\|_{H^{\f 12,0}}^2+\|(\bw_2)_\Phi\|_{H^{\f{1}{2},0}}^2+\|(\bh_1)_\Phi\|_{H^{\f 14,0}}^2\Big),
\end{align*}
and
\begin{align*}
&\f d{dt}\|h_{\Phi}\|_{L^2}^2+\lambda\|h_{\Phi}\|_{H^{\f{1}{4},0}}^2+\|\pa_yh_{\Phi}\|_{L^{2}}^2\\
&\quad\le C\Big(\|h_{\Phi}\|_{L^2}^2+\|u_{\Phi}\|_{H^{\f{1}{4},1}}^2+\|(\bw_1)_\Phi\|_{H^{\f{3}{4},0}}^2+\|(\bw_2)_\Phi\|_{H^{\f{3}{4},0}}^2\Big),\\
&\f d{dt}\|(\bh_1)_{\Phi}\|_{L^2}^2+(\lambda-C)\|(\bh_1)_{\Phi}\|_{H^{\f{1}{4},0}}^2+\|\pa_y(\bh_1)_{\Phi}\|_{L^{2}}^2\\
&\quad\le C\Big(\|h_{\Phi}\|_{L^2}^2+\|u_{\Phi}\|_{H^{\f{1}{4},1}}^2+\|(\bw_1)_\Phi\|_{H^{\f{3}{4},0}}^2+\|(\bw_2)_\Phi\|_{H^{\f{3}{4},0}}^2\Big).
\end{align*}
Thus, we can close the energy estimates in Gevrey class 2.

\section*{Acknowledgments}

Z. Zhang is partially supported by NSF of China under Grant
11371039 and 11421101.

 \end{document}